\documentclass[12pt]{article}
\usepackage[final]{epsfig}
\usepackage{graphics}
\usepackage{amsmath}
\usepackage{amsfonts}
\usepackage{latexsym}
\usepackage{amssymb}
\usepackage{graphicx}
\usepackage{url}
\usepackage{epstopdf}
\usepackage{hyperref}
\usepackage{xcolor}
\usepackage{comment}
\usepackage{mathrsfs}

\newtheorem{lemma}{Lemma}[section]
\newtheorem{proposition}[lemma]{Proposition}
\newtheorem{remark}[lemma]{Remark}
\newtheorem{example}[lemma]{Example}
\newtheorem{theorem}{Theorem}
\newtheorem{definition}[lemma]{Definition}
\newtheorem{corollary}[lemma]{Corollary}

\newcommand{\g}{{\gamma}}

\newcommand{\proofend}{$\Box$\bigskip}
\newcommand{\C}{{\mathbb C}}

\newcommand{\R}{{\mathbb R}}
\newcommand{\Z}{{\mathbb Z}}
\newcommand{\RP}{{\mathbb {RP}}}

\def\proof{\paragraph{Proof.}}

\title{Outer symplectic billiards}

\author{Peter Albers
 \and 
Ana Chavez Caliz
\and
Serge Tabachnikov
} 

\date{}

\begin{document}

\maketitle
\begin{abstract}
A submanifold of the standard symplectic space determines a partially defined, multi-valued symplectic map, the outer symplectic billiard correspondence. Two points are in this correspondence if the midpoint of the segment connecting them is on the submanifold, and this segment is symplectically orthogonal to the tangent space of the submanifold at its midpoint. This is a far-reaching generalization of the outer billiard map in the plane; the particular cases, when the submanifold is a closed convex hypersurface or a Lagrangian submanifold, were considered earlier. 

Using a variational approach, we establish the existence of odd-periodic orbits of the outer symplectic billiard correspondence. On the other hand, we give examples of curves in 4-space which do not admit 4-periodic orbits at all. If the submanifold satisfies certain conditions (which are always satisfied if its dimension is at least half of the ambient dimension) we prove the existence of two $n$-reflection orbits connecting two transverse affine Lagrangian subspaces for every $n\geq1$. In addition, for every immersed closed submanifold, the number of single outer symplectic billiard ``shots" from one affine Lagrangian subspace to another is no less than the number of critical points of a smooth function on this submanifold. 

We study, in detail, the behavior of this correspondence when the submanifold is a curve or a Lagrangian submanifold. For Lagrangian submanifolds in 4-dimensional space we present a criterion for the outer symplectic billiard correspondence to be an actual map. We show, in every dimension, that if a Lagrangian submanifold has a cubic generating function, then the outer symplectic billiard correspondence is completely integrable in the Liouville sense.
\end{abstract}

\tableofcontents

\section{Introduction} \label{sect:intro}

Outer billiard is a discrete time dynamical system defined in the exterior of an oriented plane oval (closed smooth strictly convex curve). The outer billiard map is depicted in Figure \ref{fig:out}: the point $z$ is mapped to point the $z'$ if $zz'$ is the positive tangent line of the curve at the midpoint $Q=\frac{z+z'}{2}$ of the segment $zz'$.

\begin{figure}[hbtp] 
\centering
\includegraphics[width=2.7in]{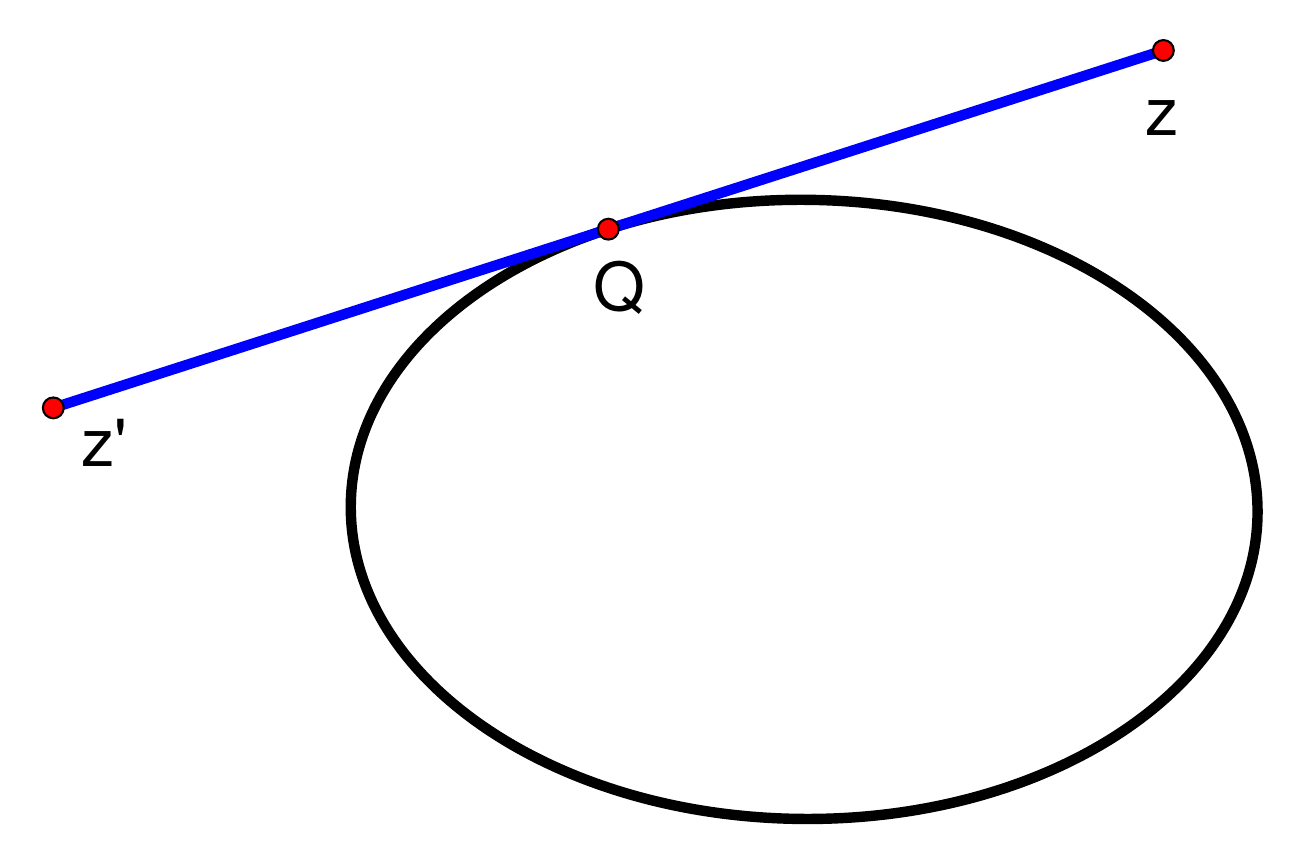}
\caption{Outer symplectic billiard map in $\R^2$.}
\label{fig:out}
\end{figure}

This dynamical system was introduced by B. Neumann \cite{Neu}; its study was put forward by J. Moser \cite{Mo1,Mo2}. In particular, the outer billiard map is area preserving and, using KAM theory, Moser showed that if the curve is smooth enough, then the outer billiard orbits stay bounded and do not escape to infinity.
Outer billiards have become a popular topic of study; see, e.g., \cite{DT} for a survey.  

The plane is a 2-dimensional symplectic space, and the definition of outer billiard was extended to the standard symplectic space in \cite{Ta1} (also see \cite{Ta3}): the outer billiard table is then a closed smooth strictly convex hypersurface, and the role of its tangent lines are played by its characteristic directions. In this way, one obtains a symplectic map of the exterior of the hypersurface. 

A version of this map (really, a multivalued map, that is, a correspondence) was studied in \cite{FT2}, where the ``table" is a Lagrangian submanifold of symplectic space. Our goal in the present paper is to study outer billiards about arbitrary immersed submanifolds in the standard symplectic space. To emphasize the symplectic nature of the problem, we call them {\it outer symplectic billiards}.

The contents of the paper are as follows. 

In Section \ref{sect:corr} we give a definition of outer symplectic billiards about a submanifold $M\subset {\R}^{2d}$ of the standard symplectic space and show that this is a symplectic correspondence.  We illustrate by an example of an ellipsoid: in this case, the correspondence is a completely integrable map. 
By way of motivation, we also provide an informal ``dictionary" from conventional billiards to outer symplectic billiards. 

For odd $n$, we present a variational approach to the study of $n$-periodic orbits of this correspondence: these orbits are represented by critical points of a certain quadratic function on $n$-gons inscribed in $M$.
We also consider outer symplectic billiard trajectories connecting two Lagrangian affine subspaces of $\R^{2d}$ in general position and present a variational approach to this problem as well. 

In Section \ref{subsect:area} we interpret the generating functions in both cases as the symplectic areas bounded by the respective polygonal lines. 

We prove in Theorem \ref{thm:existper} that, for every odd $n \ge 3$, there exists an $n$-periodic orbit of the outer symplectic billiard correspondence. In contrast, we present an example showing that this result does not extend to even values of $n$: in this (counter)example, $M$ is a loop in $\R^4$ and $n=4$.

Theorem \ref{thm:shots} concerns outer symplectic billiard trajectories connecting two Lagrangian affine subspaces in general position. We introduce a certain class of submanifolds $M$ that include all submanifolds 
of $\R^{2d}$ whose dimension is at least $d$. If $M$ belongs to this class then, for every $n\ge 1$, there exist at least two distinct outer symplectic billiard trajectories from one Lagrangian subspace to another. At present, we do not know whether this result extends to all $M \subset \R^{2d}$.

We also show that, for every closed immersed submanifold $M$, the number of single outer billiard ``shots" from one Lagrangian subspace to another is not less than the number of critical points of a smooth function on $M$.

 In Section \ref{subsect:near} we describe the symplectic outer correspondence near a {\it symplectically convex curve}; this class of curves was introduced and studied in \cite{AT}.

Section \ref{sect:lag} is devoted to a detailed study of outer symplectic billiards with respect to Lagrangian submanifolds. In particular, motivated by an example in \cite{FT2}, we study quadratic Lagrangian submanifolds in $\R^4$, given by homogeneous cubic generating functions. For such a submanifold $L^2 \subset \R^4$, we present a criterion for the outer symplectic billiard correspondence to be an actual map, defined everywhere in the complement of $L$ (Proposition \ref{prop:abcd_example_and_wall}).

We also show in Theorem \ref{thm:int} that if the Lagrangian submanifold $L^n \subset \R^{2n}$ has a cubic generating function, then the outer symplectic billiard correspondence is completely integrable: it has $n$ Poisson commuting integrals (this fact is reminiscent of the complete integrability of the usual billiards in quadratic hypersurfaces).


\bigskip

{\bf Acknowledgements}. We are grateful to A. Akopyan for a useful discussion.

PA acknowledges funding by the Deutsche Forschungsgemeinschaft (DFG, German Research Foundation) through Germany's Excellence Strategy EXC-2181/1 - 390900948 (the Heidelberg STRUCTURES Excellence Cluster), the Transregional Colloborative Research Center CRC/TRR 191 (281071066). ACC was funded by the DFG as well through Project-ID 281071066 – TRR 191.
ST was supported by NSF grants DMS-2005444 and DMS-2404535, and by a Mercator fellowship within the CRC/TRR 191. He thanks the Heidelberg University for its invariable hospitality.

We also thank the referee for useful suggestions and comments.

\section{Outer symplectic  billiard correspondence} \label{sect:corr}

\subsection{Definition of the correspondence and motivation} \label{subsect:def}

We will work in the linear symplectic space $V=\R^{2d}$ with its standard symplectic structure $\omega=\sum_{i=1}^d dx_i \wedge dy_i$. Let $M^m \subset V$ be an immersed closed submanifold of dimension $m$.

\begin{definition}
{\rm Two points of $V$, say $z$ and $z'$, are in the outer symplectic billiard correspondence with respect to $M$ if the midpoint $Q=(z+z')/2$ belongs to $M$ and the vector $z'-z$ is symplectically orthogonal to the tangent space $T_Q M$, that is,
\begin{equation}\nonumber
\begin{aligned}
Q=\tfrac12(z+z')&\in M\\
z'-z&\in T^\omega_Q M.
\end{aligned}
\end{equation}
Here, by definition, $\xi\in T^\omega_Q M$ if $\omega(\xi,\zeta)=0$ for all $\zeta\in T_Q M$.
}
\end{definition}



\begin{proposition} \label{pr:corrsymp}
The outer symplectic billiard correspondence is symplectic, i.e.,~its graph
$$
G=\{(z,z')\in V\times V\mid Q=\tfrac12(z+z')\in M,\;z'-z\in T^\omega_Q M\}
$$ 
is a Lagrangian submanifold in $(V \times V,\omega \ominus \omega)$.
\end{proposition}

\proof

Let $ (x,y)$ and $(x',y')$ be Darboux coordinates in the two copies of $V$ (so each symbol $x,y$ etc.~denotes a $d$-vector). Then  the symplectic structure in $V\times V$ is
$$
\omega \ominus \omega = d x'\wedge d y' - dx\wedge dy.
$$
Consider now the cotangent bundle $T^* V$ with coordinates $(q,q',p,p')$ (each being a $d$-vector) and the standard symplectic form $\Omega=dq\wedge dp + dq'\wedge dp'$.

One has a linear symplectic isomorphism between $V \times V$ and $T^* V$, given by the formulas
\begin{equation} \label{eq:iso}
q=\frac{x+ x'}{2},\ q'=\frac{y+ y'}{2},\ p=y'-y,\ p'= x-x'.
\end{equation}
We claim that this isomorphism maps $G \subset T^* V$ to the conormal bundle ${\mathcal N}^* M$ of $M \subset V$. 
This will suffice since the conormal bundle of any (immersed) submanifold is an (immersed) Lagrangian submanifold in $T^* V$.

Let $z=(x,y)$ and $z'=(x',y')$ be in the outer symplectic billiard correspondence, i.e.,~$(z,z')\in G$ and, in particular, $Q=(q,q')=\frac12(z+z')\in M$. Let $\xi=(u,v) \in T_Q M$ be any tangent vector, then 
\begin{equation}\nonumber
\begin{aligned}
\omega(\xi,z'-z)&=\omega((u,v),(x'-x,y'-y))\\ 
&= u \cdot (y'-y) - v \cdot (x'-x)\\ 
&= u \cdot p + v \cdot p',
\end{aligned}
\end{equation}
where $\cdot$ denotes the Euclidean inner product on $\R^d$ and the last equality is due to (\ref{eq:iso}). Hence $z'-z\in T^\omega_Q M$ if and only if the covector $P=(p,p')$ annihilates $T_{Q} M$, that is, $(q,q',p,p')$ belongs to ${\mathcal N}^* M$.
\proofend

We emphasize that the outer symplectic billiard correspondence is not a map: in general, it is multi-valued and only partially defined. In the next example, $M \subset \R^{2d}$ is an ellipsoid, and the outer symplectic billiard correspondence is a map (this holds for all closed smooth strictly convex hypersurfaces, see \cite{Ta1}). In the case of an ellipsoid, this map is also completely integrable, as we explain now.  

\begin{example} \label{ex:ellipsoid}
{\rm The normal form of a symplectic ellipsoid $M \subset \R^{2d}$ is given, in canonical coordinates, by
$$
M=\left\{(x,y)\in V\;\Big|\;\sum_{j=1}^d \frac{x_j^2+y_j^2}{a_j} =1\right\}
$$
for some $a_1,\ldots, a_n >0$. The normal direction to $M$ at $(x,y)\in M$ is
$$
N(x,y)=\left(\frac{x_1}{a_1}, \frac{y_1}{a_1},\ldots,\frac{x_n}{a_n}, \frac{y_n}{a_n}\right)\in V
$$
and the characteristic direction is
$$
J N(x,y)=\left(-\frac{y_1}{a_1}, \frac{x_1}{a_1},\ldots,-\frac{y_n}{a_n}, \frac{x_n}{a_n}\right)\in V.
$$
Given a point $(q,p)=(q_1,p_1,\ldots,q_n,p_n)$ in the exterior of $M$, one wants to find $(x,y)\in M$ such that
$$
(q,p)=(x,y)+tJ N(x,y)
$$
that is,
\begin{equation} \label{eq1}
q_j=x_j-t\frac{y_j}{a_j}, p_j=y_j+t\frac{x_j}{a_j},\ j=1,\ldots,d
\end{equation}
for some $t>0$. Once such $(x,y)$ and $t$ are found (the solution is unique due to the assumption $t>0$), one has, for the outer symplectic map $T$,
$$
(\bar q,\bar p):=T(q,p)=(x,y)-tJ N(x,y)
$$
that is,
$$
(\bar q_j,\bar p_j)=\left(x_j+t\frac{y_j}{a_j}, y_j-t\frac{x_j}{a_j}\right),\ j=1,\ldots,d.
$$
I.e., we simply change the sign of $t$ in \eqref{eq1}. From this it follows that
$$
\bar q_j^2+\bar p_j^2 = q_j^2+p_j^2,\ j=1,\ldots,d.
$$
In particular, we find $d$ integrals of the map $T$ which are Poisson-commuting (which is a straightforward computation). Therefore, the orbits of $T$ lie on Lagrangian tori $q_j^2+p_j^2=c_j,\ j=1,\ldots,d$.

Conversely, solving \eqref{eq1} for $x_j,y_j$,
$$
x_j=\frac{q_j+t\frac{p_j}{a_j}}{1+\frac{t^2}{a_j^2}},\ y_j=\frac{p_j-t\frac{q_j}{a_j}}{1+\frac{t^2}{a_j^2}},
$$
we obtain an equation for $t$:
\begin{equation*} 
\sum_{j=1}^n \frac{q_j^2+p_j^2}{a_j\left(1+\frac{t^2}{a_j^2}\right)}=1,
\end{equation*}
which has two roots with opposite signs.
}
\end{example}

\paragraph{Motivation: a loose analogy with the conventional billiards.} Denote by ${\mathcal L}$  the space of oriented lines (rays of light) in $\R^{d+1}$. This space carries a natural symplectic structure coming from symplectic reduction. However, fixing an origin and using the Euclidean structure of $\R^{d+1}$ naturally identifies ${\mathcal L}$ with $T^* S^d$, the cotangent bundle of the unit sphere. The subset of lines through a fixed point is a Lagrangian sphere in ${\mathcal L}$. In fact, the space of lines through the origin is the zero-section in $T^*S^d$.

More precisely, an oriented line is defined by its unit directing vector $q \in S^d$ and the vector $p$ dropped from the origin to the line. Since $p \perp q$, one can think of $p$ as a (co)tangent vector to $S^d$ at point $q$ (one identifies vectors and covectors by the Euclidean structure). 

We consider a smooth hypersurface $X^d \subset \R^{d+1}$ which locally can be expressed as a regular level set of a function with positive definite Hessian, in fact, a germ of such a hypersurface would be enough. The space of oriented lines which are tangent to $X$  naturally defines a hypersurface $M^{2d-1} \subset {\mathcal L}$. 

We recall that the characteristic direction of the hypersurface $M$ at a point $\ell\in M$ is, by definition, the 1-dimensional space $\ker \omega|_{T_\ell M}\subset T_\ell M$. Let us describe a generator $\xi=\xi(\ell)$ of this subspace geometrically in terms of $X$. For that, we denote by $x\in X$ the point at which the line $\ell\in M$ is tangent to $X$ and by $n=n(x)$ the normal to $X$ in $x$. 

We claim that an infinitesimal rotation of $\ell$ about the point $x$ in the plane spanned by $\ell$ and $n$ defines a tangent vector $\xi=\xi(\ell) \in T_\ell M$. Indeed, using the local convexity property of $X$, any small rotation of $\ell$ in this plane are actually (up to second order) lines tangent to $X$ at points close to $x$. Thus, $\xi(\ell)$ can be described as a tangent to an actual path in $M$ through $\ell$. An easy way of seeing this is to approximate $X$ (up to second order) by the corresponding quadratic hypersurface.
 
To show that $\xi$ is symplectically orthogonal to $T_\ell M \subset T_\ell {\mathcal L}$ we may choose the point $x$ as the origin in $\R^{d+1}$. Now, let $\eta \in T_\ell M$ be a test vector. In canonical $(q,p)$-coordinates on $\mathcal L$ the vector $\xi$ is proportional to $(n,0)$, since the $q$-part of the coordinates describes the direction of a line and the $p$-part the distance to the origin. Similarly, the $p$-component of $\eta$ lies in $T_x X$. Hence $(dq \wedge dp) (\xi,\eta)=0$, as claimed. 

In fact, the above reasoning shows that a curve in $M$ given by all tangent lines to a fixed geodesic in $X$ forms a characteristic in $M$. A more detailed exposition can be found in the lectures by J.~Moser \cite{Mo3} who uses this approach to give an elegant proof of the integrability of the geodesic flow on ellipsoids.

Next, we consider an instance of a billiard reflection in the hypersurface $X$. That is, assume that an oriented line $a$ meets $X$ at the point $x$ and is reflected to the oriented line $b$. Assuming that these lines are not orthogonal to $X$ in $x$, there is a unique line $\ell$, tangent to $X$ at point $x$, such that $a$ and $b$ are obtained from $\ell$ by rotating the same angle in the opposite directions in the plane spanned by $\ell$ and $n$. 

We now make a loose translation to the set-up of the outer symplectic billiards. The space of rays ${\mathcal L}$ becomes linear symplectic space $\R^{2d}$, and the hypersurface $M \subset {\mathcal L}$ becomes a hypersurface $M \subset \R^{2d}$. The line $\ell$ becomes a point of this $M$, and the vector $\xi$ becomes a characteristic vector to $M$ at this point. Finally, the incoming and outgoing rays $a$ and $b$ become two points on the respective characteristic line at equal distance from the tangency point of this line with $M$ (here we take the liberty of translating the equal angles condition as the equal distance one). That is, we obtain a description of the outer symplectic billiard correspondence.  

We note that the billiard reflection, considered as a transformation of the space of oriented lines, is symplectic. The above Proposition \ref{pr:corrsymp} states that the outer symplectic billiard correspondence is symplectic as well, a kind of ``sanity check" for this ``dictionary".
 
In the same spirit, the set of lines through a fixed point becomes an affine Lagrangian subspace in $\R^{2d}$. Hence, a billiard shot from point $A_1$ to point $A_2$ with $n$ reflections off $X$ translates as an $n$-link outer symplectic billiard orbit from an affine Lagrangian subspace $L_1$ to an affine Lagrangian subspace $L_2$. In what follows, we shall consider such orbits, along with $n$-periodic orbits. 

Let us also mention the case when one literally has a ``duality" between inner and outer billiards. This happens when both systems are considered in the 2-dimensional spherical geometry. In this case, one has the spherical duality between points and (oriented) great circles. This duality conjugates inner and outer billiards, exchanging length and area. See \cite{Ta1,Ta} for a detailed discussion. 

\subsection{Two Lagrangian subspaces and their generating functions} \label{subsect:gen}

We wish to consider two scenarios one might call closed and open strings. In the first one, we search for periodic orbits of the outer symplectic billiard correspondence, and in the second, orbits start and end on two affine Lagrangian subspaces. In the latter case, we assume that the Lagrangian subspaces are in general position, that is, they do not contain parallel lines. 

In this section we describe the relevant Lagrangian submanifolds and their generating functions, similar to Proposition \ref{pr:corrsymp}. In the next section, we use these functions to describe orbits of the outer symplectic correspondence. \\

Let us start with the $n$-periodic case. We argue in the spirit of Chaperon's ``g\' eod\' esiques bris\' ees'' approach to fixed points of symplectic maps \cite{Ch}, as it is presented in  \cite{Gi}; see \cite{FT1,Ta1} for the case of outer symplectic billiards. 

Under the linear symplectic isomorphism from the proof of Proposition \ref{pr:corrsymp}, 
$$
(V\times V)^{\times n} \cong (T^* V)^{\times n} = T^* (V^{\times n})
$$
consider the linear subspace  $\mathscr{L}\subset (V\times V)^{\times n}\cong T^* (V^{\times n})$ given by the equations
$
z'_i = z_{i+1},\ i=1,\ldots,n,
$
that is,
\begin{equation}\nonumber
\begin{aligned}
\mathscr{L}=\big\{(z_1,z_1',\ldots,z_n,z_n')\in (V\times V)^{\times n} \mid z'_i = z_{i+1},\; i=1,\ldots,n\big\}
\end{aligned}
\end{equation} 
where we read indices cyclically, i.e., $z_{n+1}\equiv z_1$. 

\begin{lemma} \label{lm:headtail}
The space ${\mathscr{L}}\subset T^* (V^{\times n})$ is Lagrangian. If $n$ is odd, then it has the quadratic generating function
\begin{equation} \label{eq:genfun}
F(Q_1,\ldots,Q_n) = 2 \sum_{1\le i < j \le n} (-1)^{i+j-1} \omega(Q_i,Q_j),
\end{equation}
where $Q_i = (q_i, q_i')=\frac{1}{2}(z_i+z_i')$.\\
If $n$ is even, then $\mathscr{L}$ is not given by a generating function.
\end{lemma} 

\proof
That ${\mathscr{L}} \subset (V\times V)^{\times n}$ is Lagrangian follows from the fact that the symplectic structure on $(V \times V)^{\times n}$ is $(\omega \ominus \omega)^{\oplus n}$: when evaluating this symplectic product, the equations $z'_i = z_{i+1}$ force all the terms in this long sum to cancel pairwise since we read indices cyclically.

Now let $n$ be odd.
Using (\ref{eq:iso}), we rewrite the system of linear equations $z'_i =(x_i',y_i')=(x_{i+1},y_{i+1})= z_{i+1}$, under the linear symplectic isomorphism from Proposition \ref{pr:corrsymp} as
\begin{equation} \label{eq:newsys}
p_i+p_{i+1}=2q_{i+1}'-2q_i',\quad p_i'+p_{i+1}'=2q_i - 2q_{i+1}, \ i=1,\ldots,n.
\end{equation}
Indices are again cyclic mod $n$. 

Let us treat the variables $p,p'$ as unknowns. Then, taking the alternating sum of the equations (\ref{eq:newsys}) yields the (unique) solution
$$
p_i=2 \sum_{j=1}^{n-1} (-1)^{j-1} q_{i+j}',\quad p_i'=2 \sum_{j=1}^{n-1} (-1)^{j} q_{i+j},\ i=1,\ldots,n.
$$
This uses the fact that $n$ is odd. Then rewriting 
\begin{equation}\nonumber
\begin{aligned}
F(Q_1,\ldots,Q_n) &= 2 \sum_{1\le i < j \le n} (-1)^{i+j-1} \omega(Q_i,Q_j)\\
&=2 \sum_{1\le i < j \le n} (-1)^{i+j-1} (q_i \cdot q_j' - q_i' \cdot q_j),
\end{aligned}
\end{equation}
together with a direct calculation shows that
$$
p_i=\frac{\partial F}{\partial q_i},\ p_i'=\frac{\partial F}{\partial q_i'}\;.
$$
Therefore, ${\mathscr{L}} \subset T^* (V^{\times n})$ is the graph of $dF$, i.e., is a Lagrangian submanifold with  generating function $F$.

If $n$ is even, then taking the alternating sum of the equations (\ref{eq:newsys}) shows that the condition $\sum_i (-1)^i Q_i =0$ needs to be satisfied. Thus, if this condition is satisfied, then the Lagrangian subspace ${\mathscr{L}} \subset T^* (V^{\times n})$ contains vertical directions: its projection on the base has codimension $2d$. In particular, ${\mathscr{L}}$ is not a graph of the differential of a function.
\proofend

Let us now consider the case when the boundary conditions are two affine Lagrangian subspaces $L_1$ and $L_2$ in general position. Since the outer symplectic billiard correspondence is invariant under affine symplectic transformations, we may assume that $L_1$ is the coordinate $x$-subspace and $L_2$ is the coordinate $y$-subspace of the standard linear symplectic space $V=\R^d_x \oplus \R^d_y$.

Take then the submanifold  $\mathscr{K} \subset (V\times V)^{\times n}=T^* (V^{\times n})$ given by the equations
$$
z_i'=z_{i+1},\; i=1,\ldots,n-1,\quad {\rm and}\quad z_1=(x_1,0)\in L_1,\; z_n'=(0,y_n')\in L_2,
$$
that is,
\begin{equation}\nonumber
\mathscr{K}=\left\{(z_1,z_1',\ldots,z_n,z_n')\in (V\times V)^{\times n}\; \bigg| \; 
\begin{aligned}
& z'_i = z_{i+1},\; i=1,\ldots,n-1\\ 
& z_1\in L_1, z_n'\in L_2
\end{aligned}
\right\}.
\end{equation}
We have an analog of Lemma \ref{lm:headtail}.

\begin{lemma} \label{lm:bdry}
The space $\mathscr{K}$ is Lagrangian; it is given by the generating function
\begin{equation} \label{eq:genfun1}
G(Q_1,\ldots,Q_n) = 2 \sum_{i=1}^n q_i \cdot q_i'+ 4 \sum_{1\le i < j \le n} (-1)^{j-i} q_j \cdot q_i'.
\end{equation}
\end{lemma}

\begin{remark}
{\rm
We point out that in this case the indices are not read cyclically and there is no distinction between even and odd $n$.
}
\end{remark}

\proof
Using again the linear symplectic isomorphism \eqref{eq:iso}, we rewrite our linear equations as
\begin{equation}\nonumber
\begin{aligned}
p_i+p_{i+1}&=2q_{i+1}'-2q_i', \quad &i=1,\ldots,n-1,\\ 
p_i'+p_{i+1}'&=2q_i - 2q_{i+1}, \quad &i=1,\ldots,n-1,\\[.5ex]
 p_1&=2q_1'\\
  p_n'&=2q_n.
\end{aligned}
\end{equation}
The solutions of this system are
\begin{equation}\nonumber
\begin{aligned}
p_i &= 2q_i' + 4 \sum_{1\le j<i} (-1)^{i-j} q_j',\\[.5ex]
p_i' &=2q_i + 4 \sum_{i<j\le n} (-1)^{j-i} q_j,
\end{aligned}
\end{equation}
where we point out the different range of indices under the sums. Again, a direct calculation shows that
$$
p_i=\frac{\partial G}{\partial q_i},\ p_i'=\frac{\partial G}{\partial q_i'}.
$$
Therefore, $\mathscr{K} \subset T^* (V^{\times n})$ is a Lagrangian submanifold having $G$ as  generating function.
\proofend


\subsection{Outer symplectic billiard  orbits} \label{subsect:per}

In this section, we begin to study periodic orbits of the outer symplectic billiard correspondence in $V$, as well as its orbits connecting two Lagrangian subspaces $L_1$, $L_2$ in general position. Let $M \subset V$ be an immersed closed submanifold. 

\begin{definition}\label{def:periodic_orbit_and_degenerate}
{\rm An $n$-periodic orbit of the outer symplectic billiard correspondence is an $n$-gon $(z_1,z_2,\ldots,z_n)$ in $V$ such that $z_i$ and $z_{i+1}$ are in outer symplectic billiard correspondence with respect to $M$ for each $i=1,\ldots, n$. Here, we read indices cyclically.

Likewise, an $n$-link outer symplectic billiard orbit connecting the Lagrangian subspaces $L_1$ and $L_2$ is a polygonal line  $z_1,z_2,\ldots,z_n,z_{n+1}$ in $V$ such that $z_i$ and $z_{i+1}$ are in outer symplectic billiard correspondence with respect to $M$ for each $i=1,\ldots, n$, and $z_1 \in L_1$, $z_{n+1} \in L_2$. 

Either type of orbit is called degenerate if, for some $i$, one has $z_{i-1}=z_{i+1}$ (that is, the polygon ``backtracks''), and non-degenerate otherwise.}
\end{definition}

If we set, as before,
\begin{equation} \label{eq:midpt}
Q_i = \tfrac12 (z_i + z_{i+1})\in V
\end{equation}
then backtracking is equivalent to $Q_{i-1}=Q_i$. We denote by ${\bf Q}$ the $n$-gon $(Q_1,\ldots,Q_n)\in V^{\times n}$.

Combining the results of the preceding two sections, we have

\begin{theorem} \label{thm:crit}$ $
\begin{enumerate}  
\item Let $n$ be odd. An $n$-gon $\mathbf{Z}=(z_1,z_2,\ldots,z_n)$ is an $n$-periodic orbit of the outer symplectic billiard correspondence if and only if the $n$-gon ${\bf Q}$ is inscribed in $M$ and is a critical point of the function $F$ restricted to $M^{\times n}\subset V^{\times n}$ given by the expression \eqref{eq:genfun}.  
\item A polygonal line $\mathbf{Z}=(z_1,z_2,\ldots,z_n,z_{n+1})$ is an $n$-link outer symplectic billiard orbit connecting the Lagrangian subspaces $L_1$ and $L_2$ if and only if the $n$-gon ${\bf Q}$ is inscribed in $M$ and is a critical point of the function $G$ restricted to $M^{\times n}\subset V^{\times n}$ given by \eqref{eq:genfun1}.
\end{enumerate}
\end{theorem}

\proof
The argument of the proof of Proposition \ref{pr:corrsymp} shows that periodic orbits and $n$-links correspond to intersection points between the conormal bundle ${\mathcal N}^* (M^{\times n})  \subset T^* (V^{\times n})$ and the Lagrangian subspaces $\mathscr{L}$ and $\mathscr{K}$, respectively. Lemma \ref{lm:headtail} (and similarly Lemma \ref{lm:bdry}) implies that these intersection points are precisely the critical points of the function $F$ (and similarly $G$) restricted to $M^{\times n}$. 
\proofend

Theorem \ref{thm:crit} has the following immediate consequence.

\begin{corollary} \label{cor:oneshot}
The number of 1-link outer symplectic billiard orbits connecting the Lagrangian subspaces $L_1$ and $L_2$ is greater than or equal to the number of critical points of a smooth function on $M$ (in particular, not less that the Lusternik-Schnirelmann category of $M$, and if $M$ is generic, not less than its sum of Betti numbers).
\end{corollary} 

\proof
In the case when $n=1$, the matter of backtracking cannot occur, and thus all orbits are nondegenerate, and they correspond to the critical points of the restriction of the function $G(Q)= q \cdot q'$ to $M$. 
\proofend

\begin{example} \label{ex:all}
{\rm 
The function $G(Q)=q \cdot q'$ may vanish identically on $M$. In this case, every point of $M$ is the midpoint of a 1-link outer symplectic billiard orbit connecting Lagrangian subspaces $L_1$ and $L_2$. Here is an example of such $M^2 \subset \R^4$, a symplectic torus, parameterized by two angles $\alpha$ and $\beta$:
$$
q_1=\cos\alpha\cos\beta,\ q_1'=\sin\alpha\sin\beta,\ q_2=\sin\alpha\cos\beta,\ q_2'=-\cos\alpha\sin\beta.
$$
}
\end{example}

\subsection{Symplectic meaning of the generating functions} \label{subsect:area}

Let us recapitulate the main point of the preceding section: to an outer symplectic billiard orbit $\mathbf{Z}=(z_1,z_2,\ldots,z_n)$ we assign the polygonal line ${\bf Q}$ made of the midpoints  $Q_i = \tfrac12 (z_i + z_{i+1})$, a closed polygon in the case of $n$-periodic orbits and a polygonal line for the orbits connecting two Lagrangian subspaces in general position (in the former case, we assume that $n$ is odd). 
The orbits correspond to the critical points of the functions $F({\bf Q})$ and $G({\bf Q})$, restricted to $M^{\times n}$. 
We now reveal the symplectic meaning of these functions. 

Let ${\bf Z}=(z_1,\ldots,z_n,z_{n+1})$ be an outer symplectic billiard orbit; in the periodic case, one has $z_{n+1}=z_1$. Let 
$$
A({\bf Z}) = \frac{1}{2} \sum_{i=1}^n \omega(z_i,z_{i+1})
$$
be the symplectic area bounded by this orbit.

\begin{lemma} \label{lm:revareas}
Depending on whether the orbit in question is periodic or connecting two Lagrangian subspaces, one has
$$
A({\bf Z})=F({\bf Q}) \ \ {\rm and}\  \ A({\bf Z})= G({\bf Q})
$$
where we assume that $n$ is odd in the periodic case.
\end{lemma}

\proof The proof consists of substituting 
$$
q_i=\tfrac12 (x_i + x_{i+1}), q_i'=\tfrac12 (y_i + y_{i+1})
$$
 into the double sums $F({\bf Q})$ or $G({\bf Q})$ and collecting terms, resulting in the single sum 
$$
\frac{1}{2} \sum_{i=1}^n (x_i \cdot y_{i+1} - x_{i+1} \cdot y_i)
$$
which is the expression $A(\mathbf{Z})$.
\proofend

We note that, in the periodic case, the symplectic area $A({\bf Z})$ of a polygon is invariant under affine symplectic transformations, and in the case of the boundary conditions on the coordinate Lagrangian subspaces $\R^d_x$ and $\R^d_y$, the symplectic area $A({\bf Z})$ is invariant under linear symplectomorphisms that preserve the polarization $V=\R^d_x \oplus \R^d_y$. It follows that the generating functions $F({\bf Q})$ and $G({\bf Q})$ are equivariant relative to the respective symplectic transformations.

We also note that in the periodic case, $A({\bf Z})$ (and hence $F({\bf Q})$) is invariant under the cyclic permutation of its arguments and it changes sign if the cyclic order is reversed.

The following lemma describes how to reconstruct the polygonal line ${\bf Z}$ from the midpoint polygon ${\bf Q}$. 

\begin{lemma} \label{lm:rec}
$ $
\begin{enumerate}
\item In the periodic case, if $n$ is odd, then an $n$-gon ${\bf Q}$ uniquely determines an $n$-gon ${\bf Z}$ such that ${\bf Q}$ is the mid-point polygon of ${\bf Z}$.
If $n$ is even, then ${\bf Q}$ is the mid-point polygon of some polygon ${\bf Z}$ if and only if
\begin{equation} \label{eq:alt}
\sum_{i=1}^n (-1)^i Q_i =0.
\end{equation}
However, in this case ${\bf Z}$ is not unique: an arbitrary choice of a vertex $z_1$ uniquely determines a polygon ${\bf Z}=(z_1,\ldots,z_n)$ with ${\bf Q}$ as its midpoint polygon.
\item In the case of the boundary conditions on two Lagrangian subspaces in general position, an $n$-gon ${\bf Q}$ uniquely determines an $n$-gon ${\bf Z}=(z_1,\ldots, z_{n+1})$ such that ${\bf Q}$ is the mid-point polygon of ${\bf Z}$ and $z_1\in L_1, z_{n+1} \in L_2$.
\end{enumerate}
\end{lemma}

\proof
Consider the periodic case first. 
We start with the following observation. Let ${\bf Z}=(z_1,z_2,\ldots,z_n)$ be any polygon and ${\bf Q}=(Q_1,\ldots,Q_n)$ the corresponding $n$-gon of midpoints. By definition, the reflection in the point $Q_i$ takes $z_i$ to $z_{i+1}$. Now, consider the composition of the reflections in $Q_1,\ldots, Q_n$. By construction, $z_1$ is a fixed point of this composition. 

Let ${\bf Q}=(Q_1,\ldots,Q_n)$ be an arbitrary $n$-gon and consider the above composition of reflections. Assume that $n$ is odd and recall that the composition of an odd number of reflections is again a reflection, thus has a unique fixed point, say $z_1$. In this case, set ${\bf Z}=(z_1,z_2,\ldots,z_n)$ to be the orbit of $z_1$ under the sequence of reflections which is $n$-periodic since $z_1$ is the unique fixed point of the $n$-fold composition.

Now, let $n$ be even. Recall that the composition of an even number of reflections is a parallel translation. Choose any point $z_1$ and again determine $z_2,\ldots,z_n, z_{n+1}$ by consecutive reflections in the vertices of ${\bf Q}$. Then summing up the relations \eqref{eq:midpt} shows that $z_{n+1}=z_1$ is equivalent to \eqref{eq:alt}. In other words, the obtained parallel translation has a fixed point (and thus is the identity) if and only if \eqref{eq:alt} holds. In particular, the consecutive reflections close up for one initial point if and only if they close up for any initial point.

Next, consider the boundary conditions on two Lagrangian subspaces in general position. As before, without loss of generality, we assume that these subspaces are $\R^d_x$ and $\R^d_y$. We need to solve the system of equations
\begin{equation}\nonumber
\begin{aligned}
x_i + x_{i+1}&=2q_i, \quad &i=1,\ldots,n,\\ 
y_i+y_{i+1}&=2q_i', \quad &i=1,\ldots,n,\\[.5ex]
y_1&=0,\\ 
x_{n+1}&=0.
\end{aligned}
\end{equation}
This system has 
\begin{equation}\nonumber
\begin{aligned}
x_i&=2\sum_{j=i}^n (-1)^{j-i} q_j,\quad &i=1,\ldots,n,\\[.5ex]
y_i&= 2 \sum_{j=1}^{i-1} (-1)^{i-j+1} q_j',\quad &i=2,\ldots,n+1,
\end{aligned}
\end{equation}
together with $y_1=x_{n+1}=0$ as unique solution. This concludes the proof.
\proofend

\subsection{Existence and non-existence of periodic orbits} \label{subsect:exist}


In this section, we establish the existence of non-degenerate outer symplectic billiard orbits. First, we consider the periodic case.

\begin{theorem} \label{thm:existper}
Let $M\subset V$ be an immersed closed submanifold. 
Then for every odd $n\geq3$ there exists a non-degenerate $n$-periodic orbit of the outer symplectic billiard correspondence.
\end{theorem}

\proof
Let $n$ be odd and consider $F:M^{\times n}\to\R$ from above. Below, we will also consider $F:M^{\times (n+2)}\to\R$. Both functions, by abuse of notation, are labeled by the letter $F$. 

According to Theorem \ref{thm:crit}, a critical point ${\bf Q}$ of $F$ gives, for $n$ odd, a polygon ${\bf Z}$ which is an $n$-periodic outer symplectic billiard orbit. However, recall that periodic orbits might be degenerate, see Definition \ref{def:periodic_orbit_and_degenerate}, i.e., ${\bf Q}$ might have consecutive points which coincide, corresponding to backtracking in the polygon ${\bf Z}$ associated to ${\bf Q}$ by Lemma \ref{lm:rec}.

First, consider the case that $M$ is contained in some affine Lagrangian subspace of $V$. Since the function $F$ is invariant under translations when $n$ is odd, we may assume that $M$ lies in a linear Lagrangian subspace $L\subset V$. But then all terms of the sum (\ref{eq:genfun}) that comprises $F({\bf Q})$ vanish, that is, $F$ is identically zero. Therefore every point in  $M^{\times n}$ is critical, and they all represent periodic orbits. In particular, one has infinitely many non-degenerate $n$-periodic orbits. Argued more geometrically, pick any $n$-gon $\mathbf{Q}$ on $M$ and construct the corresponding $n$-gon $\mathbf{Z}$ according to Lemma \ref{lm:rec}. Then $\mathbf{Z}$ is indeed an $n$-periodic outer symplectic billiard orbit since $z_{i+1}-z_{i}\in L=T_{Q_i}L\subset T_{Q_i}^\omega M$.

\medskip

Now we assume that $M$ is not contained in an affine Lagrangian subspace of $V$, i.e., condition (L)
\begin{equation}\label{eq:assumption_on_M}\tag{L}
\forall v\in V\;\forall L\subset V \text{ linear Lagrangian subspace}: v+M\not\subset L
\end{equation}
is satisfied. We claim that a maximum ${\bf Q}=(Q_1,\ldots, Q_n)\in M^{\times n}$ of $F$ does not exhibit backtracking, i.e., for all $i$ we have $Q_i\neq Q_{i+1}$. Thus, the corresponding polygon ${\bf Z}$ is a non-degenerate $n$-periodic orbit.

To show this, we first examine how the value of $F({\bf Q})$ changes when appending two points from $M$ to ${\bf Q}$. That is, if ${\bf Q}=(Q_1,\ldots, Q_n)$ and $\overline {\bf Q}=(Q_1,\ldots, Q_n, Q_{n+1}, Q_{n+2})$ with $Q_i \in M$, then
\begin{equation}\nonumber
\begin{aligned}
F(&\overline {\bf Q})-F({\bf Q})\\
&=\sum_{1\le i < j \le n+2} (-1)^{i+j-1} \omega(Q_i,Q_j)-\sum_{1\le i < j \le n} (-1)^{i+j-1} \omega(Q_i,Q_j)\\
&=\sum_{i=1}^n(-1)^{i+n+1-1} \omega(Q_i,Q_{n+1})+\sum_{i=1}^{n+1}(-1)^{i+n+2-1} \omega(Q_i,Q_{n+2})\\
&=\sum_{i=1}^n(-1)^{i+n} \omega(Q_i,Q_{n+1})+\sum_{i=1}^{n}(-1)^{i+n+1} \omega(Q_i,Q_{n+2})+\omega(Q_{n+1},Q_{n+2})\\
&=\omega\left(\sum_{i=1}^n(-1)^{i+n} Q_i, \; Q_{n+1}-Q_{n+2}\right)+\omega(Q_{n+1},Q_{n+2})\\
&=\omega\left(-\sum_{i=1}^n(-1)^{i} Q_i, \; Q_{n+1}-Q_{n+2}\right)+\omega(Q_{n+1},{Q}_{n+2}).
\end{aligned}
\end{equation}
Assume now that the polygon $\overline {\bf Q}=(Q_1,\ldots, Q_{n+2})\in M^{\times (n+2)}$ has consecutive points which coincide. Since $F$ is invariant under cyclic permutations, we may assume that $Q_{n+1}=Q_{n+2}$. 
By the above computation, excising the points $Q_{n+1}$ and $Q_{n+2}$ does not change the value of $F$, that is, $F(\overline {\bf Q})=F({\bf Q})$. 

On the other hand, we claim that, using \eqref{eq:assumption_on_M}, we can find two points $Q'_{n+1}, Q'_{n+2}\in M$ such that $ {\bf Q}' = (Q_1, \ldots , Q_n, Q'_{n+1}, Q'_{n+2})$ satisfies $F({\bf Q}')> F({\bf Q})$ (and similarly for ``$<$'').

Indeed, if we abbreviate $v:=-\sum_{i=1}^n(-1)^{i} Q_i$,  then 
\begin{equation}\nonumber
\begin{aligned}
F({\bf Q}')-F({\bf Q})&=\omega(v, Q'_{n+1}-Q'_{n+2})+\omega(Q'_{n+1},Q'_{n+2})\\
&=\omega(-v+Q'_{n+1},-v+Q'_{n+2}).
\end{aligned}
\end{equation}
Now, $-v+M$ is not contained in a linear Lagrangian subspace, by assumption \eqref{eq:assumption_on_M}, and therefore we can find two points $-v+x, -v+y\in-v+M$ (i.e.,~$x,y\in M$) with $\omega(-v+x,-v+y)\neq0$. Without loss of generality, we may assume that $\omega(-v+x,-v+y)>0$. We conclude that
\begin{equation}\nonumber
F({\bf Q}')>F({\bf Q}) = F(\overline{\bf Q})
\end{equation}
if we choose $Q'_{n+1}=x$ and $Q'_{n+2}=y$. Exchanging the roles of $x$ and $y$ gives $F({\bf Q}')<F({\bf Q})$.

Now assume that ${\bf Q}=(Q_1,\ldots, Q_{n})\in M^{\times n}$ is a maximum of $F$. 
If ${\bf Q}$ has consecutive points which coincide then, as demonstrated above, we find another $n$-gon ${\bf Q}'$, obtained by replacing the consecutive points in ${\bf Q}$, satisfying $F({\bf Q}')> F({\bf Q})$. This directly contradicts that ${\bf Q}$ is a maximum of $F$. Therefore, a maximum of $F$ cannot contain consecutive points which coincide. In other word, the maximum of $F$ represents a non-degenerate $n$-periodic orbit. The same works for a minimum of $F$. Unfortunately, reversing the order of ${\bf Q}$ turns a maximum of $F$ into a minimum of $F$, and we can therefore only guarantee to find one non-degenerate $n$-periodic orbit.
\proofend

\begin{remark} \label{rmk:growth}
{\rm For a generic $M$, one expects the number of non-degenerate periodic orbits to grow linearly in the period. For planar Birkhoff billiards (billiards inside ovals), this follows from Birkhoff's theorem: for every ratio $0<\frac{p}{q}\leq\frac12$ in lowest terms, there exist at least two $q$-periodic billiard trajectories with the rotation number $p$. 
For Euclidean billiards in strictly convex domains with smooth boundaries in higher dimensional spaces, the number of periodic trajectories grows linearly with the period and with the dimension; see
 \cite{FT1}. 
}
\end{remark}

\begin{remark} \label{rmk:L}
{\rm
Here are two classes of examples of closed, immersed $M\subset V$ that satisfy condition $\eqref{eq:assumption_on_M}$:

\begin{itemize}
\item If $m=\dim M\geq \tfrac12\dim V$ then $M$ cannot be contained in an affine Lagrangian subspace.
\item For $m=1$, let $M\equiv \gamma$ be symplectically convex, i.e.,~$\omega(\gamma'(t),\gamma''(t))>0$ for all $t$ (see \cite{AT}). If $v+\gamma\subset L$ for some linear Lagrangian subspace $L\subset V$, then $\omega(\gamma'(t),\gamma''(t))=0$ for all $t$ since $\gamma'(t), \gamma''(t) \in L$.
\end{itemize}
}
\end{remark}

The following example, suggested by A.~Akopyan, exhibits curves that do not admit non-degenerated $4$-periodic orbits of the outer symplectic billiard correspondence. This shows that the case of even periods needs a different approach. 

\begin{lemma} \label{lm:Cheb}
The outer symplectic billiard correspondence with respect to the Chebyshev curve $\gamma:\R/{2\pi\Z}\to\R^4$ 
$$
\gamma(t)=\big(\cos t, \sin t, \cos (2t), \sin (2t)\big) \in \R^4
$$
does not have non-degenerated 4-periodic orbits. (The components of $\gamma$ comprise a Chebyshev system of functions.)
\end{lemma}

\proof
Let $(z_1, \ldots, z_4)$ be a non-degenerate $4$-periodic orbit.
First, we observe that Lemma \ref{lm:rec} asserts, in particular, that the midpoints $Q_i=\frac12(z_i+z_{i+1}) \in \gamma$  satisfy 
$$
Q_1 -Q_2 +Q_3 -Q_4 =0.
$$ 
That is, $\gamma$ has an inscribed parallelogram $(Q_1, Q_2, Q_3, Q_4)$. Since the orbit is non-degenerate, see Definition \ref{def:periodic_orbit_and_degenerate}, this parallelogram has distinct consecutive vertices. However, opposite vertices may coincide, that is, it is possible that $Q_1=Q_3$ or $Q_2=Q_4$, but not both -- otherwise the parallelogram degenerates to a point. Thus, we also need to consider degenerate parallelograms consisting of a segment with three points on it.

First, we show that $\gamma$ does not admit inscribed planar quadrilaterals at all, in particular, no non-degenerate parallelograms. 

Indeed, a non-trivial trigonometric polynomial of degree $d$ has at most $2d$ different roots in the interval $[0, 2\pi)$. In particular, any affine hyperplane in $\R^4$ intersects $\gamma$ in at most $4$ different points, unless $\gamma$ is contained in this hyperplane. Clearly, $\gamma$ is not contained in any affine hyperplane. 

Now assume that $\gamma$ admits an inscribed planar quadrilateral $(Q_1, \ldots, Q_4)$. Then we can choose another point $P$ on the curve $\g$ not contained in the affine 2-dimensional plane spanned by the planar quadrilateral $(Q_1, \ldots, Q_4)$ since $\g$ is not contained in an affine hyperplane. Then the hyperplane containing $Q_1, \ldots, Q_4, P$ intersects $\gamma$ in at least $5$ points, a contradiction. 

It remains to consider the case when the parallelogram degenerates. If $Q_2=Q_4$, but $Q_1 \neq Q_3$, then the curve $\gamma$ contains three distinct points that lie on a line. In this case, we can choose two points $P_1$ and $P_2$ on $\g$, so that the five points $Q_1,Q_2,Q_3,P_1,P_2$ span a hyperplane that intersects $\gamma$ in at least $5$ points, again a contradiction, and similarly in the other case. This completes the proof. 
\proofend

\begin{remark}
{\rm
From the above proof it is clear that any curve having the property that it intersects any hyperplane in at most 4 points does not have non-degenerated 4-periodic orbits.

Similarly, the above argument generalizes to Chebyshev curves of degree $n\geq3$, e.g., $\gamma:\R/{2\pi\Z}\to\R^{2n}$,
$$
\gamma(t)= (\cos t, \sin t, \cos (2t), \sin (2t),\ldots, \cos(nt), \sin(nt)) \in \R^{2n}.
$$
Such a curve does not admit a $2n$-periodic outer symplectic billiard orbit with all reflection points $Q_1,\ldots, Q_{2n}$ being pairwise distinct.

Indeed, by Lemma \ref{lm:rec}, $\sum (-1)^i Q_i =0$, therefore the points $Q_1,\ldots,Q_{2n}$ lie in an affine subspace of codimension 2. Then one can choose a point $P \in \g$ so that the hyperplane spanned by $Q_1,\ldots, Q_{2n}, P$ has at least $2n+1$ intersections with $\g$, a contradiction.

However, this does not preclude periodic orbits with fewer than $2n$ distinct reflection points. For example, by Theorem \ref{thm:existper}, there always exists a 3-periodic outer symplectic billiard orbit. Traversing it twice, yields a 6-periodic orbit for every curve $\g \in \R^{2n}$, in particular for Chebyshev curves. 
}
\end{remark}

Next, we consider the existence of non-degenerate $n$-link outer symplectic billiard ``shots" from one Lagrangian subspace to another one. For $n=1$, Corollary \ref{cor:oneshot} already provides the best possible Morse- or category-theoretical lower bound. 

To tackle the general case, let us introduce a class of closed immersed submanifolds $M \subset V$ satisfying the condition (LL)
\begin{equation} \label{eq:good}\tag{LL}
\forall P\in V\ \exists x\in M : T_x M \not\subset (x-P)^\omega.
\end{equation}
Here, $(x-P)^\omega$ denotes the symplectic complement of the linear space $\R(x-P)$ inside $V$. Notice that condition (LL) is invariant under translations of $M$. 

\begin{example} \label{ex:Leg}
{\rm Here is a class of examples of manifolds that fail condition (\ref{eq:good}). Let $P$ be the origin and $S^{2d-1} \subset V$ the unit sphere centered in $P$ equipped with its standard contact structure. That is, the contact hyperplane $\xi_x$ at the point $x \in S^{2d-1}$ is the kernel of the Liouville 1-form $i_x \omega$. 

Let $M \subset (S^{2d-1},\xi)$ be an isotropic submanifold, i.e., for every $x\in M$ we have $T_xM\subset\xi_x$. Such $M$ fails condition (LL); for $P$ the origin, and any $x\in M$, $T_x M \subset \xi_x \subset x^\omega$.
}
\end{example}

It follows from the Frobenius theorem that isotropic submanifolds $M$ of a contact manifold satisfy $\dim M\leq d-1$. The following lemma shows that this dimension condition is necessary to fail condition (LL).

\begin{lemma} \label{lm:adm}
Let $M\subset V^{2d}$ be a closed immersed manifold. If $\dim M \geq d$, then $M$ satisfies condition $\mathrm{(LL)}$.
\end{lemma}

\proof
Let us prove the equivalent statement: if $M$ fails condition (LL), then $\dim M \leq d-1$. 

The assumption means that there exists a point $P$ such that for all $x \in M$ one has $T_xM \subset (x-P)^\omega$. Without loss of generality, assume that $P$ is the origin.

First, we claim that there exists $0\neq x \in M$ such that $T_x M$ does not contain the vector $x$. Indeed, otherwise the radial vector field would be tangent to $M$, and $M$ would be conical, contradicting its compactness.

Consider a small neighborhood of such a point $x$ and project it radially on the unit sphere $S^{2d-1}$. This projection is a diffeomorphism and, as explained in Example \ref{ex:Leg} above, its image is tangent to the contact distribution in $S^{2d-1}$. The dimension of such isotropic submanifold does not exceed $d-1$. This implies the result.
\proofend

We now proceed to an analog of Theorem \ref{thm:existper} for outer symplectic billiard ``shots" connecting two Lagrangian subspaces in general position.

\begin{theorem} \label{thm:shots}
Let $M\subset V$ be a closed immersed submanifold satisfying condition $\mathrm{(LL)}$ and $n\ge 2$. Then there exist at least two distinct non-degenerate $n$-link outer symplectic billiard trajectories connecting $L_1$ and $L_2$. (The case of $n=1$ was already covered in Corollary \ref{cor:oneshot}).
\end{theorem}

\proof
The proof is analogous to that of Theorem \ref{thm:existper} and we omit some calculations. 

The trajectories under consideration correspond to the critical points of the function $G: M^{\times n} \to \R$. We claim that the maximum and minimum of $G$ are $n$-gons ${\bf Q}$ that do not backtrack.

Let us argue for a maximum; the case of a minimum is similar. We argue by contradiction. Let ${\bf Q}=(Q_1,\ldots,Q_n)$ be a maximum point such that $Q_i=Q_{i+1}$ for some $i$. Deleting these two points from the $n$-tuple of points of $M$ does not change the value of the function $G$ (as follows from Lemma \ref{lm:revareas}) and from \eqref{eq:dif} below.

We want to add two new points, $Q_{i}$ and $Q_{i+1}$, in such a way that the value of the function $G$  increases. Set
$$
\overline G = G(Q_1,\ldots,Q_{i-1},Q_{i+2},Q_n),\ G=G(Q_1,\ldots,Q_{n}).
$$
Then one calculates
\begin{equation} \label{eq:dif}
G - \overline G = 2 w \cdot (q'_{i+1}-q'_{i}) - 2(q_{i+1}-q_i)\cdot w' + 2q_{i}\cdot q_i'+ 2 q_{i+1}\cdot q_{i+1}' - 4 q_{i+1}\cdot q_{i}',
\end{equation}
where $w, w'$ are $d$-dimensional vectors given by
$$
w= 2 \sum_{j=i+2}^{n} (-1)^{i-j+1} q_j, \ w'= 2 \sum_{j=1}^{i-1} (-1)^{i-j} q_j'.
$$
If we introduce the vectors 
$$
W:=(w,w'),\ U:=(q_{i+1},q_{i}')= Q_{i} + (q_{i+1}-q_{i},0)\in V
$$
we may rewrite \eqref{eq:dif} as
$$
G - \overline G = 2\omega(W+U,Q_{i+1}-Q_{i}).
$$
We are ready to choose the points $Q_{i}$ and $Q_{i+1}$. For that, we apply condition (LL) to the point $P=-W$. That is, we 
find $x\in M$ such that 
$$
T_x M \not\subset  (x+W)^\omega.
$$
In particular, there exists $\xi \in T_x M$ such that $\omega(x+W,\xi)=1$. 

If we set $Q_{i}:=x$ and choose $Q_{i+1}$ to be a point on the geodesic starting at $Q_{i}$ in  direction $\xi$ which has distance $\varepsilon > 0$ from $Q_{i}$, 
we conclude
$$
U = Q_{i} + O(\varepsilon),\ Q_{i+1}-Q_{i} = \varepsilon \xi + O(\varepsilon^2),
$$
therefore
$$
\omega(W+U,Q_{i+1}-Q_{i}) = \varepsilon \omega(W+Q_{i},\xi) + O(\varepsilon^2) = \varepsilon + O(\varepsilon^2).
$$
It follows that, choosing a sufficiently small positive $\varepsilon$, we make $G - \bar G$ positive. This is a desired contradiction.
\proofend

Theorem \ref{thm:shots} is weaker than one would like it to be: we think that its result holds for all submanifolds $M$. And, similarly to Remark \ref{rmk:growth}, we expect the number of $n$-link non-degenerate orbits to grow linearly with $n$.

\subsection{Outer symplectic billiard map near a curve} \label{subsect:near}

In this section, we consider the outer symplectic billiard relation in the case of a curve $\gamma$. We repeat the definition given in the previous section.

\begin{definition}
{\rm    
A curve $\gamma: \R \to V$ in the symplectic vector space $(V,\omega)$ is called symplectically convex if  $\omega(\gamma'(t), \gamma''(t)) > 0$ for any $t\in \R$.
}
\end{definition}

This notion was introduced in \cite{AT}. In this section, all curves will be assumed to be symplectically convex, and we typically parametrize a symplectically convex curve $\gamma$ so that $\omega(\gamma'(t),\gamma''(t))=1$. Also, we assume that $\g$ is a closed curve, i.e.,~$\g:S^1\to\R^{2n}$. Let $X\subset \R^{2d}\times S^1$ be given by 
$$
X=\{(v,t) \mid \omega(v,\gamma'(t))=0\}=\bigcup_{t\in S^1} T_{\gamma(t)}^\omega\gamma\times\{t\}.
$$ 
Since $\gamma'(t)\neq0$ for all $t$, this is a smooth immersed submanifold of codimension 1 of $\R^{2d}\times S^1$.

Consider the map $\Psi:X\to\R^{2d}$ given by $\Psi(v,t)=\gamma(t)+v$. The image $\Psi(X)\subset\R^{2d}$ of this map is the set where the outer symplectic billiard relation with respect to $\gamma$ is defined. At the regular values of $\Psi$, the multiplicity function is defined: it is the locally constant function taking values in non-negative integers that counts the number of preimages of $\Psi$. How this number changes is, of course, of interest for the symplectic billiard relation. Therefore, let us define the set $\Delta\subset X$ of critical points and the set $\Sigma:=\Psi(\Delta)\subset\R^{2d}$ of critical values of $\Psi$. In particular, the number of preimages is locally constant on the set $\Psi(X)\setminus\Sigma\subset\R^{2d}$. We call $\Sigma$ the ``wall'' (one needs to have crossed it if the multiplicity changed). The following proposition describes properties of $\Delta$ and $\Sigma$ if the curve $\gamma$ is symplectically convex.
 

\begin{proposition}\label{prop:properties_Delta_Sigma_symp_convex_curve}
Let $\gamma \subset \R^{2d}$ be a closed, symplectically convex curve. Then, $\Delta=\mathrm{Crit}(\Psi)\subset X$ is a submanifold of $\dim\Delta=2d-1$ and can be described as 
\begin{equation}\label{eq:Delta}
\Delta=\{(v,t)\in X\mid \omega(v,\gamma''(t))=0\}.
\end{equation}
Moreover, the wall $\Sigma=\Psi(\Delta)$ can be expressed as
\begin{equation}\label{eq:wall}
\Sigma=\left\{P\in\R^{2d}\;\Bigg| \;\;\exists t\in S^1\text{ with }\;
\begin{aligned}
\omega(P,\gamma'(t))&=\omega(\gamma(t),\gamma'(t))\\ 
\omega(P,\gamma''(t))&=\omega(\gamma(t),\gamma''(t)) 
\end{aligned}
\right\}.
\end{equation}
If we define the set $\Delta^{\mathrm{sing}}\subset\Delta$ by
\begin{equation}\nonumber
\Delta^{\mathrm{sing}}:=\left\{(v,t)\in\Delta\;\Big| \;\; \omega(v,\gamma'''(t))=\omega(\gamma'(t),\gamma''(t))\right\}
\end{equation}
and $\Sigma^{\mathrm{sing}}\subset\Sigma$ by
\begin{equation} \label{eq:third}
\Sigma^{\mathrm{sing}}:=\left\{P\in\Sigma\;\Bigg| \;\;
\begin{aligned}
&\exists t\in S^1\text{ with }\\
&\omega(P,\gamma'''(t))=\omega(\gamma'(t),\gamma''(t)) +\omega(\gamma(t),\gamma'''(t)) 
\end{aligned}
\right\}
\end{equation}
then $\Psi(\Delta^{\mathrm{sing}})=\Sigma^{\mathrm{sing}}$ and
\begin{equation}\nonumber
\Psi|_{\Delta\setminus\Delta^{\mathrm{sing}}}:\Delta\setminus\Delta^{\mathrm{sing}}\to\Sigma\setminus\Sigma^{\mathrm{sing}}
\end{equation}
is a local diffeomorphism. In fact, $\Delta^{\mathrm{sing}}$ is precisely the set where $D\Psi|_{T\Delta}$ has not full rank (that is, its rank is strictly less than $2d-1$, the dimension of $\Delta$). 
\end{proposition}

\begin{remark}
{\rm
We point out that for the sake of readability the definition \eqref{eq:third} of $\Sigma^{\mathrm{sing}}$ is not quite accurate. The point part ``$P\in\Sigma$'' already includes a choice of $t\in S^1$. This is the same $t$ as in the additional equation. That is, $P\in \Sigma^{\mathrm{sing}}$ if there exists $t\in S^1$  with
\begin{equation}\label{eq:third_all_eqns}
\left\{
\begin{aligned}
\omega(P,\gamma'(t))&=\omega(\gamma(t),\gamma'(t))\\ 
\omega(P,\gamma''(t))&=\omega(\gamma(t),\gamma''(t))\\
\omega(P,\gamma'''(t))&=\omega(\gamma'(t),\gamma''(t)) +\omega(\gamma(t),\gamma'''(t)) 
\end{aligned}
\right.
\end{equation}
} 
\end{remark}

\proof

Recall  that the image of the map $\Psi:X\to\R^{2d}$, given by $\Psi(v,t)=\gamma(t)+v$, is the set where the outer symplectic billiard relation with respect to $\gamma$ is defined. By definition, $P\in {\rm Im} \Psi\subset \R^{2d}$ if and only if there exists $(v,t)\in X$ such that $P=\gamma(t)+v$ with $v\in T_{\gamma(t)}^\omega\gamma$, that is
\begin{equation} \label{eq:first}
\omega(P,\gamma'(t))=\omega(\gamma(t),\gamma'(t)).
\end{equation} 
For a fixed $t$, this equation describes the affine hyperplane $\gamma(t)+T_{\gamma(t)}^\omega\gamma$ and explains the first equation in \eqref{eq:wall}. Now, we determine $\Delta$. To do so, we observe that for $(v,t) \in X$,

\begin{equation}\nonumber
T_{(v,t)}X=\{(\hat v,\hat t)\in\R^{2d}\times\R\mid \omega(\hat v,\gamma'(t))+\hat t\omega(v,\gamma''(t))=0\}
\end{equation}
and
\begin{equation}\nonumber
D\Psi(v,t)(\hat v,\hat t)=\hat t\gamma'(t)+\hat v\in \R^{2d}.
\end{equation}
Now assume that $(v,t)$ satisfies $\omega(v,\gamma''(t))=0$. Then 
\begin{equation}\nonumber
T_{(v,t)}X=\{(\hat v,\hat t)\mid \omega(\hat v,\gamma'(t))=0\}=(\R\gamma'(t))^\omega\times\R
\end{equation}
and 
\begin{equation}\nonumber
D\Psi(v,t)(T_{(v,t)}X)=(\R\gamma'(t))^\omega\subsetneq\R^{2d}.
\end{equation}
In particular, $D\Psi(v,t)$ is not surjective. This shows the first inclusion in \eqref{eq:Delta}. 

Now assume that $\omega(v,\gamma''(t))\neq0$. Using that $\gamma$ is symplectically convex, we can split 
$$
\R^{2d}=(\R\gamma'(t))^\omega\oplus(\R\gamma''(t)).
$$
For $w\in (\R\gamma'(t))^\omega$ set $(w,0)\in T_{(v,t)}X$ and then $D\Psi(v,t)(w,0)=w$. For $w=\gamma''(t)$ we consider
\begin{equation}\nonumber
\left\{
\begin{aligned}
\hat t&:=\frac{1}{\omega(v,\gamma''(t))},\\
\hat v&:=\gamma''(t)-\hat t\gamma'(t)=\gamma''(t)-\frac{1}{\omega(v,\gamma''(t))}\gamma'(t).
\end{aligned}
\right.
\end{equation}
Then $(\hat v,\hat t)\in T_{(v,t)}X$, using $\omega(\g'(t),\g''(t))=1$. Moreover, 
\begin{equation}\nonumber
D\Psi(v,t)(\hat v,\hat t)=\hat t\gamma'(t)+\hat v=\gamma''(t).
\end{equation}
Thus, $\omega(v,\gamma''(t))\neq0$ implies that $D\Psi(v,t)$ is surjective and, therefore, the other inclusion in \eqref{eq:Delta} is established.

Now, we show that $\Delta$ is a manifold. For that, we consider the map 
\begin{equation}\nonumber
\begin{aligned}
\phi: \R^{2d}\times S^1 &\to \R^2\\
(v,t)&\mapsto\phi(v,t)= \big(\omega(v,\gamma'(t)),\;\omega(v,\gamma''(t))\big)
\end{aligned}
\end{equation}
and check that $(0,0)\in\R^2$ is a regular value of $\phi$. Indeed, already the map
\begin{equation}\nonumber
\R^{2d}\ni\xi\mapsto D\phi(v,t)[\xi,0]=\big(\omega(\xi,\gamma'(t)),\omega(\xi,\gamma''(t))\big)\in\R^2
\end{equation}
is onto if $\gamma$ is symplectically convex, since
\begin{equation}\nonumber
D\phi(v,t)[\gamma'(t),0]=(0,1)\quad\text{and}\quad D\phi(v,t)[\gamma''(t),0]=(-1,0).
\end{equation}
The claim $\Psi(\Delta^{\mathrm{sing}})=\Sigma^{\mathrm{sing}}$ follows by observing that the three equations determining $(v,t)\in\Delta^{\mathrm{sing}}$ directly translate into the three equations determining $(P,t)\in\Sigma^{\mathrm{sing}}$ via the bijection $(v,t)\mapsto (\Psi(v,t)=\gamma(t)+v,t)=(P,t)$. A similar argument applies to the description of $\Sigma$ given by equation (\ref{eq:wall}).

It remains to show that
\begin{equation}\nonumber
D\Psi(v,t):T_{(v,t)}\Delta\to \R^{2d}
\end{equation}
does not have full rank if and only if $(v,t)\in\Delta^{\mathrm{sing}}$. Let $(v,t)\in\Delta$, i.e., 
\begin{equation}\label{eq:vt_in_Delta}
\left\{
\begin{aligned}
\omega(v,\gamma'(t))&=0\\
\omega(v,\gamma''(t))&=0.
\end{aligned}
\right.
\end{equation}
Then, we have
\begin{equation}\label{eq:hat_vt_in_TDelta}
T_{(v,t)}\Delta=\left\{
(\hat v, \hat t)\in\R^{2d}\times\R\;\;\Bigg|\;\;
\begin{aligned}
&\omega(\hat v,\gamma'(t))=0\\
&\omega(\hat v,\gamma''(t))+\hat t\omega(v,\gamma'''(t))=0
\end{aligned}
\right\}.
\end{equation}
Now, assume that $(\hat v,\hat t)\in\ker D\Psi(v,t)$, that is, $\hat t\gamma'(t)+\hat v=0$ or
\begin{equation}\nonumber
\hat v=-\hat t\gamma'(t).
\end{equation}
Combining this with the second equation in \eqref{eq:hat_vt_in_TDelta} we conclude
\begin{equation}\nonumber
-\hat t\omega(\gamma'(t),\gamma''(t))+\hat t\omega(v,\gamma'''(t))=0.
\end{equation}
Notice that if $\hat t=0$ then $\hat v=0$. Thus, if $(\hat v,\hat t)\neq0$ then
\begin{equation}\label{eq:vt_in_Delta_sing}
\omega(v,\gamma'''(t))=\omega(\gamma'(t),\gamma''(t)).
\end{equation}
We showed that if $\ker D\Psi(v,t)\neq\{0\}$ then $(v,t)\in\Delta$ satisfies the additional equation \eqref{eq:vt_in_Delta_sing}, i.e., $(v,t)\in \Delta^{\mathrm{sing}}$. 

Conversely, assume $(v,t)\in \Delta^{\mathrm{sing}}$, i.e.,
\begin{equation}\nonumber
\left\{
\begin{aligned}
\omega(v,\gamma'(t))&=0\\
\omega(v,\gamma''(t))&=0\\
\omega(v,\gamma'''(t))&=\omega(\gamma'(t),\gamma''(t))
\end{aligned}
\right.
\end{equation}
and consider $(\hat v,\hat t):=(\gamma'(t),-1)$. Then
\begin{equation}\nonumber
(\hat v,\hat t)\in T_{(v,t)}\Delta
\end{equation}
and
\begin{equation}\nonumber
D\Psi(v,t)(\hat v,\hat t)=0,
\end{equation}
i.e., $\ker D\Psi(v,t)\neq\{0\}$. Thus, we showed that $\Delta^{\mathrm{sing}}$ is precisely the set on which $D\Psi|_\Delta$ does not have full rank.  
\proofend

\begin{remark}
{\rm
While the above proof consists of fairly straightforward computation, the equations for $\Delta$, $\Sigma$ etc.~in Proposition \ref{prop:properties_Delta_Sigma_symp_convex_curve} have concrete geometric meaning. Let us explain this a bit further.

Recall that $P\in {\rm Im} \Psi\subset \R^{2d}$ if and only if there exists $t\in S^1$ such that
\begin{equation} \nonumber
\omega(P,\gamma'(t))=\omega(\gamma(t),\gamma'(t)).
\end{equation} 
That is, if $P$ is contained in the affine hyperplane $\gamma(t)+T_{\gamma(t)}^\omega\gamma$.

The intersection of an infinitesimally close pair of such hyperplanes is then given by the set of points $P$ for which there exists $t\in S^1$ such that the equations
\begin{equation} \nonumber
\begin{aligned}
&\omega(P,\gamma'(t))=\omega(\gamma(t),\gamma'(t)),\\ 
&\omega(P,\gamma''(t))=\omega(\gamma(t),\gamma''(t))
\end{aligned}
\end{equation} 
hold. This is also the location where the multiplicity function should change its value.

Likewise, the intersection of an infinitesimally close triple of such hyperplanes, the singular points of $\Sigma$, is
given by the set of points $P$ for which there exists $t\in S^1$ such that
\begin{equation} \nonumber
\begin{aligned}
&\omega(P,\gamma'(t))=\omega(\gamma(t),\gamma'(t)),\\ 
&\omega(P,\gamma''(t))=\omega(\gamma(t),\gamma''(t)),\\
&\omega(P,\gamma'''(t))=\omega(\gamma'(t),\gamma''(t)) +\omega(\gamma(t),\gamma'''(t))
\end{aligned}
\end{equation} 
hold. This slightly heuristic argument directly gives the descriptions of the sets $\Sigma$ and $\Sigma^{\mathrm{sing}}$.

A slightly different geometric point of view is the following. We observe that equation \eqref{eq:first} can be rewritten as
\begin{equation}\label{eq:image_of_Psi}
(J\gamma'(t))\cdot (P-\gamma(t))=0
\end{equation}
where $J$ is the complex structure on $\R^{2d}\cong\C^d$. If we set 
$$
c(t):=\Big(J\gamma'(t),-\big(J\gamma'(t)\big)\cdot\gamma(t) \Big)\in\R^{2d}\times\R
$$
and $\bar P:=(P,1)\in\R^{2d}\times\R$, then the affine equation \eqref{eq:image_of_Psi} becomes the linear equation
\begin{equation}\label{eq:first_rephrase}
c(t)\cdot \bar P=0.
\end{equation}
Now, fixing $P\in\R^{2d}$, our multiplicity function asks for the number of $t\in S^1$ such that equation \eqref{eq:first} holds with that given $P$. We may rephrase this as the number of $t\in S^1$ such that \eqref{eq:first_rephrase} holds, i.e.,~how often the curve $c$ intersects the hyperplane $(\R\bar P)^\perp\subset\R^{2d}\times\R$. Thus, if the counting function changes its value at a point $P$ then this  intersection cannot be transverse, that is, 
$$
c'(t)\cdot \bar P=0.
$$
Using
\begin{equation}\nonumber
\begin{aligned}
c'(t)
&=\Big(J\gamma''(t),-\big(J\gamma''(t)\big)\cdot\gamma(t)\Big),
\end{aligned}
\end{equation}
a non-transverse intersection implies
\begin{equation}\nonumber
0=c'(t)\cdot \bar P=(J\gamma''(t))\cdot (P-\gamma(t))=\omega(\gamma''(t),P-\gamma(t)),
\end{equation}
i.e., equations \eqref{eq:wall} are satisfied.  Similarly, the equation 
$$
c''(t)\cdot\bar P=0
$$
is equivalent to the additional equation \eqref{eq:third}, i.e., the intersection is not only not transverse but of even higher order.
} 
\end{remark}

Next, we observe that near a symplectically convex curve the wall ``has no singularities''. In particular, the wall $\Sigma$ is nearby an honest codimension-1 submanifold.

\begin{lemma} \label{lm:near}
A  curve $\g$ has a neighborhood that does not contain points of $\Sigma^{\mathrm{sing}}$ if and only if it is symplectically convex. Moreover, in this is the case, 
$$
\g\subset\Sigma \quad\text{and}\quad T_{\g(t)}\Sigma=T_{\gamma(t)}^{\omega}\gamma
$$ 
for all $t\in S^1$.
\end{lemma}

\begin{proof}
The third equation in \eqref{eq:third_all_eqns} cannot hold in a sufficiently small neighborhood of $\gamma$. Indeed, using symplectic convexity, we may rewrite it as
$\omega(P-\gamma(t),\gamma'''(t))=1$ and since the left-hand side is small for $P$ close to $\gamma$ we arrive at a contradiction.

Conversely, assume that symplectic convexity fails at one point $t_0$, i.e., we have $\omega(\g'(t_0),\g''(t_0))=0$ and $\g'(t_0)\neq0$. Then, $P=\gamma(t_0)$ trivially satisfies the three equations in \eqref{eq:third_all_eqns}, that is, $P \in \Sigma^{\mathrm{sing}}$.

For the second claim, we observe that $P=\g(t)$ clearly satisfies both equations in \eqref{eq:wall}, i.e.,~$\g\subset\Sigma$. Finally, $\Sigma$ is ruled by the spaces $\mathrm{Span}(\gamma'(t),\gamma''(t))^{\omega}$, hence 
$$T_{\gamma(t)} \Sigma = {\rm Span} (\gamma'(t),(\gamma'(t),\gamma''(t))^{\omega}) = (\R\gamma'(t))^{\omega}= T_{\gamma(t)}^{\omega}\gamma,$$
as claimed.
\proofend
\end{proof}

\begin{example} \label{ex:34}
{\rm
Consider a cubic parabola in the plane, Figure \ref{fig:34}. 
\begin{figure}[hbt] 
\centering
\includegraphics[height=2in]{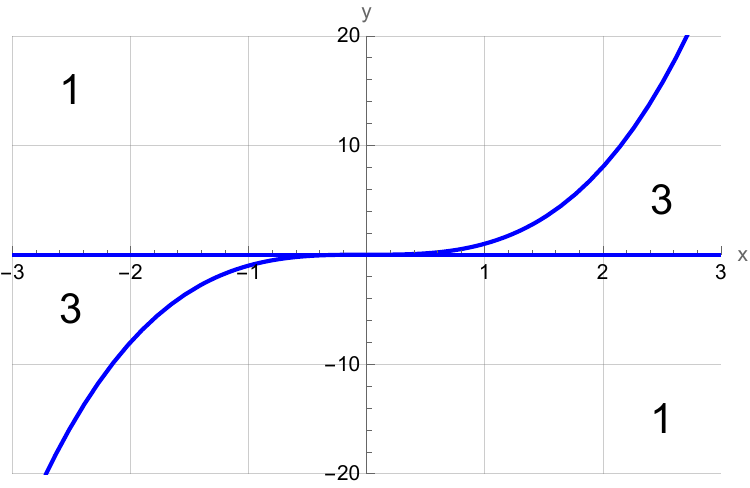}\\[2ex]
\includegraphics[height=2in]{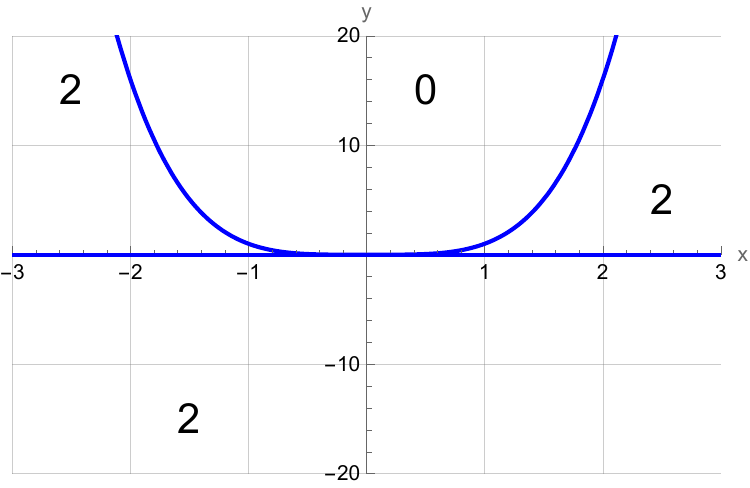}
\caption{A cubic and a quartic parabolas.}
\label{fig:34}
\end{figure}
It fails to be (symplectically) convex at the origin. $\Sigma$ is the union of the curve and the horizontal axis, and $\Sigma^{\mathrm{sing}}$ is the origin. $\Sigma$ separates the plane into four regions; the multiplicity equals 3 on the convex, and 1 on the concave sides of the parabola. The multiplicity does not change when crossing at the origin.

For a quartic parabola, $\Sigma$ also is the union of the curve and the horizontal axis, but the whole horizontal axis is $\Sigma^{\mathrm{sing}}$. The multiplicity is 2 on the convex and 0 on the concave sides of the parabola; it does not change when crossing the axis except at the origin.
}
\end{example}



Now, we come to the main observation for symplectically convex curves. The wall divides a neighborhood of the curve into two ``sides''. We will show that the multiplicity function has values $0$ resp.~$2$ on either side. This is exactly the behavior for outer billiards for a strictly convex curve in the plane: there are precisely two tangent rays emanating from a point outside the curve, and there is no such ray if the point is contained inside the curve.

\begin{lemma} \label{lm:0or2}
The local multiplicity in a sufficiently small neighborhood of a symplectically convex curve $\gamma$ is either 0 or 2. More precisely, a point $P$ in a sufficiently small neighborhood of $\gamma$ can either not be written as $\gamma(t)+v$ with $\omega(\gamma'(t),v)=0$, or can be written in this way in precisely two different ways -- always assuming $P$ being close to $\gamma(t)$ and $v$ being small.
\end{lemma}

\begin{proof}
From the previous Lemma \ref{lm:near}, we already know that the wall $\Sigma$ is a smooth hypersurface near $\gamma$. We claim that the multiplicity on either side of $\Sigma$ is 0 or 2 as long as we are near $\gamma$. 

As discussed above, the multiplicity is locally constant since it can only change when crossing $\Sigma$. Therefore, it is enough to exhibit a point on each side of $\Sigma$ (arbitrarily close to $\gamma$): one with multiplicity 0 and the other with multiplicity 2. 

Our candidate is the point $P:=\gamma(0)+\delta\gamma''(0)$ where, without loss of generality, we choose $t=0$. As we will show, for all sufficiently small $\delta>0$, there is no way to express $P=\gamma(t)+v$ with $\omega(\gamma'(t),v)=0$, $t\neq0$, and $v\neq0$ small. On the other hand, for all sufficiently small $\delta<0$, there are precisely two different such representations.

For that, we consider the function
\begin{equation}\label{eqn:def_eta}
\eta(t) =\frac{\omega(\gamma(t)-\gamma(0),\gamma'(t))}{\omega(\gamma''(0),\gamma'(t))}
\end{equation}
which is well-defined for $t$ close to $0$, again due to symplectic convexity. 
We will show below that
\begin{equation}
\eta(t)=-\tfrac12t^2+O(t^3).
\end{equation}
Now assume that for some small $\delta\in\R$ we can write $P=\gamma(0)+\delta\gamma''(0)$ as 
\begin{equation}\nonumber
P=\gamma(0)+\delta\gamma''(0)=\gamma(t)+v
\end{equation}
with $\omega(\gamma'(t),v)=0$, where $t\neq0$ and $v\neq0$ are both small. Thus,
\begin{equation}\nonumber
\delta\g''(0)=\g(t)+v-\g(0)
\end{equation}
and we obtain
\begin{equation}\nonumber
\delta=\frac{\omega(\gamma(t)+v-\gamma(0),\gamma'(t))}{\omega(\gamma''(0),\gamma'(t))}=\frac{\omega(\gamma(t)-\gamma(0),\gamma'(t))}{\omega(\gamma''(0),\gamma'(t))}
\end{equation}
using $\omega(\gamma'(t),v)=0$. By assumption, $\delta$ and $t$ are small and therefore $\delta=\eta(t)$. Again, since $t$ is small, $\eta(t)=-\tfrac12t^2+O(t^3)$ is strictly negative. Therefore, if $\delta>0$ is positive, we arrive at the desired contradiction. 

On the other hand, if $\delta<0$ then there are precisely two small values $t_-<0<t_+$ with $\eta(t_\pm)=\delta$ and we claim that $P:=\gamma(0)+\delta\gamma''(0)$ has precisely two presentations as $P=\gamma(t)+v$ with $\omega(\gamma'(t),v)=0$ and $t\neq0$ and $v\neq0$ small. Indeed
\begin{equation}\nonumber
P=\gamma(0)+\delta\gamma''(0)=\gamma(t_\pm)+v_\pm
\end{equation}
with 
\begin{equation}\nonumber
v_\pm:=\gamma(0)+\underbrace{\eta(t_\pm)}_{=\delta} \gamma''(0) - \gamma(t_\pm).
\end{equation}
Finally, since $t_\pm$ is small, so is $v$ and we see, using equation \eqref{eqn:def_eta}, that
\begin{equation}\nonumber
\begin{aligned}
\omega(\gamma'(t_\pm),v_\pm)&=\omega(\gamma'(t_\pm),\gamma(0)+\eta(t_\pm) \gamma''(0) - \gamma(t_\pm))\\
&=\omega(\gamma'(t_\pm),\gamma(0) - \gamma(t_\pm)) +\eta(t_\pm) \omega(\gamma'(t_\pm),\gamma''(0))\\
&=\omega(\gamma'(t_\pm),\gamma(0) - \gamma(t_\pm))\\&\qquad\qquad +\frac{\omega(\gamma(t_\pm)-\gamma(0),\gamma'(t_\pm))}{\omega(\gamma''(0),\gamma'(t_\pm))}\omega(\gamma'(t_\pm),\gamma''(0))\\
&=0.
\end{aligned}
\end{equation}
It remains to show that $\eta(t)=-\tfrac12t^2+O(t^3)$ near $t=0$. For that, we use Taylor expansion of $\gamma(t)$ and $\gamma'(t)$ and obtain
\begin{equation}\nonumber
\begin{aligned}
\eta(t)&=\frac{\omega(\gamma(t)-\gamma(0),\gamma'(t))}{\omega(\gamma''(0),\gamma'(t))}\\
&=\frac{\omega(\gamma'(0)t+\tfrac12\gamma''(0)t^2 +O(t^3),\gamma'(0)+\gamma''(0)t+\tfrac12\gamma'''(0)t^2+O(t^3))}{\omega(\gamma''(0),\gamma'(0)+\gamma''(0)t+\tfrac12\gamma'''(0)t^2+O(t^3))}\\
&=\frac{t^2-\tfrac12t^2+O(t^3)}{-1+\tfrac12\omega(\gamma''(0),\gamma'''(0))t^2+O(t^3)}\\
&=-\tfrac12t^2+O(t^3),
\end{aligned}
\end{equation} 
as required. \proofend 
\end{proof}


\section{Outer symplectic billiard correspondence for Lagrangian submanifolds} \label{sect:lag}

Motivated by examples studied by Fuchs--Tabachnikov in \cite{FT2}, we consider now the outer symplectic billiard correspondence for Lagrangian submanifolds in $\R^{2n}$. 

\subsection{The wall}\label{sec:Lagrangian_case_the_wall}

Let $L\subset \R^{2n}$ be a Lagrangian submanifold. We first discuss the definition of the wall, in analogy to the curve case, see Section \ref{subsect:near}. Consider the map
\begin{equation}\label{eqn:Psi_Lagrangian_case}
\begin{aligned}
\Psi:TL&\to \R^{2n}\\
(X,W)&\mapsto X+W
\end{aligned}
\end{equation}
where we recall that $T^\omega L=TL$. The image of $\Psi$ is the domain on which the outer symplectic billiard correspondence is defined. 

\begin{definition}
Let $L\subset \R^{2n}$ be a Lagrangian submanifold. We denote by $\Delta\subset TL$ the set of critical points of $\Psi$ and by $\Sigma=\Psi(\Delta)\subset\R^{2n}$ the set of singular values of the map $\Psi$. The set $\Sigma\subset\R^{2n}$ is called the wall. 
\end{definition}

\begin{remark}
{\rm
It turns out that in the Lagrangian case, the wall may be very ``thin'', i.e.,~may not be of codimension 1, see examples below.
}
\end{remark}

To simplify the discussion, we assume from now on that $L$ is globally given by a generating function $F:\R^n\to\R$, more precisely, 
\begin{equation}\nonumber
L=\big\{(q_1,\ldots,q_n, F_1(q),\ldots,  F_n(q))\mid q=(q_1,\ldots,q_n)\in\R^n\big\}\subset\R^{2n},
\end{equation}
where $F_i=\frac{\partial F}{\partial q_i}$. For short, we write 
$$
L=\{(q,\nabla F(q))\mid q\in \R^n\}.
$$
The symplectic structure at $(q,p)\in \R^{2n}$ is $\omega=\sum_idp_i\wedge dq_i$.
The tangent space of $L$ can be written as
\begin{equation}\nonumber
\begin{aligned}
T_{(q,\nabla F(q))}L&=\{(v,\mathrm{Hess}F(q)v)\mid v\in\R^n\}\\
&=\{(v_1,\ldots,v_n,\sum_iF_{1i}(q)v_i,\ldots,\sum_iF_{ni}(q)v_i)\}.
\end{aligned}
\end{equation}
To simplify subsequent notation, let us introduce the following:
\begin{equation}\nonumber
\begin{aligned}
\nabla^2F(x)v&:=\mathrm{Hess}F(x)v=\left(\sum_iF_{1i}(x)v_i,\ldots,\sum_iF_{ni}(x)v_i\right)\\
\nabla^3F(x)[v,\zeta]&:=\left(\sum_{ij}F_{1ij}(x)v_i\zeta_j,\ldots,\sum_{ij}F_{nij}(x)v_i\zeta_j\right),\\
\end{aligned}
\end{equation}
that is, $\nabla^2F(x):\R^n\to\R^n$ and $\nabla^3F(x):\R^n\times\R^n\to\R^n$. With this notation, we may write
\begin{equation}\nonumber
L=\{(q,\nabla F(q))\mid q\in \R^n\},\quad T_{(q,\nabla F(q))}L=\{(v,\nabla^2F(q)v)\mid v\in\R^n\}.
\end{equation}
%
%
%
Using the generating function $F$, the map $\Psi$, see \eqref{eqn:Psi_Lagrangian_case}, can be expressed as
\begin{equation}\nonumber
\begin{aligned}
\Psi:\R^n\times\R^n&\to \R^{2n}\\
(q,w)&\mapsto (q+w,\nabla F(q)+\nabla^2F(q)w).
\end{aligned}
\end{equation}
Note that the zero section $\mathcal{O}\subset TL$ corresponds to the set $\{w=0\}\subset\R^n\times\R^n$ and the latter space really should be read as $T^*\R^n$. We will not distinguish between $\Psi:TL\to\R^{2n}$ and $\Psi:\R^n\times\R^n\to\R^{2n}$. The differential of $\Psi$ at $(q,w)$ is
\begin{equation}\label{eqn:derivative_Psi}
\begin{aligned}
D\Psi(q,w):\R^n\times\R^n&\to \R^{2n}\\[1ex]
(\zeta,\eta)&\mapsto 
(\zeta+\eta,\nabla^2 F(q)(\zeta+\eta)+\nabla^3F(q)[\zeta,w]).
\end{aligned}
\end{equation}
Thus, by definition, $\Delta$, resp.~the wall $\Sigma$, is the set of points $(q,w)\in\R^n\times\R^n$, resp.~the set of points $\Psi(q,w)\in\R^{2n}$, for which the linear map
\begin{equation}\nonumber
\R^n\times\R^n\ni(\zeta,\eta) \mapsto (\zeta+\eta,\nabla^2 F(q)(\zeta+\eta)+\nabla^3F(q)[\zeta,w])\in \R^{2n}
\end{equation}
is not surjective or, equivalently, not injective. 

Before we give a characterization of $\Delta$ and the wall $\Sigma$, let us compare the Lagrangian case with the curve case, i.e., the special situation where $L=\gamma\subset\R^2$ is a one-dimensional Lagrangian submanifold with generating function $F:\R\to\R$, that is,
\begin{equation}\nonumber
\g(q)=(q,F'(q)), \quad \g'(q)=(1,F''(q)), \quad\g''(q)=(0,F'''(q))\quad\forall q\in\R
\end{equation}
and
\begin{equation}\nonumber
\nabla^2F(q)w=wF''(q), \quad \nabla^3F(q)[w,\zeta]=w\zeta F'''(q)\quad\forall w,\zeta\in\R.
\end{equation}
The derivative of the map 
$$\Psi(q,w)=(q+w,F'(q)+wF''(q))=\g(q)+w\g'(q)$$
 is
\begin{equation}\nonumber
D\Psi(q,w)[\zeta,\eta]=(\zeta+\eta)\g'(q)+w\zeta\g''(q)=(\zeta+\eta,(\zeta+\eta)F''(q)+w\zeta F'''(q)).
\end{equation}
Setting $\zeta=1=\eta$, we see that $D\Psi(q,w)[1,1]=2\g'(q)+w\g''(q)$, as expected. Moreover, $D\Psi(q,w)$ is surjective if and only if $w\neq0$ and $\g'(q),\g''(q)$ are not collinear. The latter means that $\g$ has no inflection points or, equivalently, that $\g$ is symplectically convex since we are in dimension 2. 

In the general Lagrangian case, we have the following characterization of the wall in terms of a generating function $F$.
\begin{lemma}\label{lem:description_of_wall_Lagrangian_case}
The set $\Delta$ of singular points of $\Psi:\R^n\times\R^n\to \R^{2n}$ is 
\begin{equation}\nonumber
\begin{aligned}
\Delta=\big\{(q,w)\in\R^n\times\R^n\mid\R^n\ni\zeta\mapsto \nabla^3F(q)[\zeta,w]\in\R^n\text{ is singular}\,\big\}.
\end{aligned}
\end{equation}
Therefore, a point $\Psi(q,w)=(q+w,\nabla F(q)+\nabla^2F(q)w)\in\R^{2n}$ is an element of $\Sigma$ if and only if the linear map
\begin{equation}\nonumber
\R^n\ni\zeta\mapsto \nabla^3F(q)[\zeta,w]\in\R^n
\end{equation}
does not have full rank. 
\end{lemma}

\begin{remark}
{\rm
In particular, $\{w=0\}\subset\Delta$. This corresponds to the fact that the zero section $\mathcal{O}\subset TL$ is contained in $\Delta$ and therefore, $L\subset\Sigma\subset\R^{2n}$. In certain very specific examples described below, we actually have $L=\Sigma$.
} 
\end{remark}

\begin{proof}
The map $\nabla^3F(q)[\cdot,w]$ vanishes along  $\{w=0\}$ and, thus, is certainly singular.

We have to show that $D\Psi(q,w):\R^n\times\R^n\to \R^{2n}$, see \eqref{eqn:derivative_Psi}, is not surjective, or equivalently, not injective if and only if $\R^n\ni\zeta\mapsto \nabla^3F(q)[\zeta,w]\in\R^n$ is singular. For that assume that there exist $\zeta,\eta\in\R^n$ with
\begin{equation}\nonumber
(\zeta+\eta,\nabla^2 F(q)(\zeta+\eta)+\nabla^3F(q)[\zeta,w])=(0,0).
\end{equation}
We immediately conclude that $\zeta+\eta=0$ and the condition reduces to
\begin{equation}\nonumber
\nabla^3F(q)[\zeta,w]=0,
\end{equation}
that is, $\zeta\mapsto \nabla^3F(q)[\zeta,w]$ has non-trivial kernel. Conversely let $\zeta$ be in the kernel of $\nabla^3F(q)[\cdot,w]$ then $D\Psi(q,w)[\zeta,-\zeta]=0$. \proofend 
\end{proof}

\paragraph{Relation to the Lagrangian Grassmannian.}
We make a brief excursion and arrive at Lemma \ref{lem:description_of_wall_Lagrangian_case} from a more differential geometric perspective. Let us consider again a Lagrangian embedding $j:L\hookrightarrow\R^{2n}$ (or even an immersion) and denote by $\Lambda_n$ the Grassmannian of (unoriented) Lagrangian subspaces of $\R^{2n}$. Then $j$ induces a natural map
\begin{equation}\nonumber
\begin{aligned}
G:L&\to\Lambda_n\\
x&\mapsto T_xL.
\end{aligned}
\end{equation}
We set $\ell:=T_xL$ for some fixed $x\in L$ and consider the differential of $G$, i.e.,
\begin{equation}\nonumber
D_xG:\ell\to T_\ell\Lambda_n.
\end{equation}
Using the symplectic form on $\R^{2n}$ we have an identification
\begin{equation}\nonumber
T_\ell\Lambda_n\stackrel{\omega}{=}\{A\in\mathrm{Hom}(\ell,\ell^*)\mid A=A^*\}=S^2\ell^*\subset(\ell^*)^{\otimes 2}.
\end{equation}
We may therefore regard the linear map $D_xG:\ell\to T_\ell\Lambda_n\subset(\ell^*)^{\otimes 2}$ as an element
\begin{equation}\nonumber
D_xG\in(\ell^*)^{\otimes 3}.
\end{equation}
In terms of a generating function $F$ for $L$ this is identified with
\begin{equation}\nonumber
D_xG=\nabla^3F(q)
\end{equation}
where $x=(q,\nabla F(q))$. We recall that $\nabla^3F(q):\R^n\times\R^n\to\R^n$ is, as a third derivative, indeed a tri-linear form on $\R^n$ with values in $\R$ (in disguise). This gives the expression $\nabla^3F(q)$ a more differential geometric meaning.

\subsection{Three examples in dimension 4} \label{subsect:3examples}

We begin by discussing three instructive examples. The \textbf{first example} is the remarkable Example 3.2 / 3.5 from \cite{FT2} of a quadratic Lagrangian submanifold $L\subset \R^4$. This Lagrangian submanifold is given by the homogeneous cubic generating function 
\begin{equation}\nonumber
F(q_1,q_2):=q_1^2q_2+q_1q_2^2 : \R^2\to\R,
\end{equation} 
that is,
\begin{equation}\nonumber
L=\{(q,\nabla F(q))\mid q\in\R^2\}=\{(q_1,q_2,2q_1q_2+q_2^2,q_1^2+2q_1q_2)\mid(q_1,q_2)\in\R^2\}.
\end{equation}
For convenience, let us collect the various derivatives of $F$.
\begin{equation}\nonumber
\begin{aligned}
\nabla F(q)&=(2q_1q_2+q_2^2,q_1^2+2q_1q_2),\\[2ex]
\nabla^2 F(q)w&=\big(2q_2w_1+2(q_1+q_2)w_2, 2(q_1+q_2)w_1+2q_1w_2\big)\\
&=
\begin{pmatrix}
2q_2 & 2(q_1+q_2)\\
2(q_1+q_2) & 2q_1 
\end{pmatrix}
\begin{pmatrix}
w_1\\
w_2
\end{pmatrix},\\[2ex]
\nabla^3F(q)[w,\zeta]&=\big(2w_1\zeta_2+2w_2(\zeta_1+\zeta_2), 2w_1(\zeta_1+\zeta_2)+2w_2\zeta_1 \big)\\
&=\begin{pmatrix}
2w_2 & 2(w_1+w_2)\\
2(w_1+w_2) & 2w_1 
\end{pmatrix}
\begin{pmatrix}
\zeta_1\\
\zeta_2
\end{pmatrix},\\[2ex]
\nabla^4 F&\equiv 0.
\end{aligned}
\end{equation}
Since $F$ is cubic we have the identity $\nabla^3F(q)[w,\zeta]=\nabla^2F(w)\zeta$. 

As in \cite{FT2}, we show that, for the example $F=q_1^2q_2+q_1q_2^2$, the wall $\Sigma$ is as small as possible, namely $\Sigma=L$. We begin by determining the set $\Delta$ of singular points of $\Psi$.
For that recall from Lemma \ref{lem:description_of_wall_Lagrangian_case} that $(q,w)\in\Delta$ if and only if
\begin{equation}\nonumber
\R^n\ni\zeta\mapsto \nabla^3F(q)[\zeta,w]\in\R^n
\end{equation}
does not have full rank. 
This map is independent of $q$ and is described by the matrix
\begin{equation}\nonumber
\begin{pmatrix}
2w_2 & 2(w_1+w_2)\\
2(w_1+w_2) & 2w_1 
\end{pmatrix}.
\end{equation}
Its determinant (dropping a factor 4)
\begin{equation}\nonumber
w_1w_2-(w_1+w_2)^2=-w_1^2-w_1w_2-w_2^2
\end{equation}
vanishes if and only if $w_1=w_2=0$. In other words, in this case $\Delta=\{w=0\}$, and thus $\Sigma=L$ is as small as possible. We will see below that, in this example, the outer symplectic billiard correspondence gives rise, after an initial choice, to a map.

\bigskip

We contrast this with the \textbf{second example}. Here, let $L$ be an affine Lagrangian subspace of $\R^4$. Then the map $\Psi(X,W)=X+W$ takes $TL$ into $L$. In particular, $\Delta=TL$ and $\Sigma=L$. Here again, the wall is as small as possible, but for completely different reasons. The set $\Delta$ is maximal and the outer symplectic billiard correspondence does not gives rise to a well-defined map.

\bigskip

We present a \textbf{third example} before treating the general cubic case in dimension 4. Consider $F(q)=q_1^3+q_2^3$ and the corresponding Lagrangian 
\begin{equation}\nonumber
L=\{(q,\nabla F(q))\mid q\in\R^2\}.
\end{equation}
Then, up to reordering the coordinates, $L\cong\gamma\times\gamma$ is the product of two copies of the curve
\begin{equation}\nonumber
\g:=\{(t,3t^2)\mid t\in \R\}\subset\R^2.
\end{equation}
We again list the derivatives
\begin{equation}\nonumber
\begin{aligned}
\nabla F(q)&=(3q_1^2,3q_2^2),\\
\nabla^2F(q)w&=(6q_1w_1,6q_2w_2),\\
\nabla^3F(q)[\zeta,w]&=(6w_1\zeta_1,6\zeta_2w_2).
\end{aligned}
\end{equation}
In particular, $\zeta\mapsto\nabla^3F(q)[\zeta,w]=\nabla^2F(w)\zeta$ is described by the matrix
\begin{equation}\nonumber
\begin{pmatrix}
6w_1 & 0\\
0 & 6w_2
\end{pmatrix}.
\end{equation}
Its determinant $36w_1w_2$ vanishes if and only if $w_1=0$ or $w_2=0$. The case $w_1=0$ corresponds to points in $\Sigma$ of the form
\begin{equation}\nonumber
\begin{aligned}
(q_1,q_2+w_2,\nabla F(q)+\nabla^2F(q)[0,w_2])&=(q_1,q_2+w_2,3q_1^2,3q_2^2+6q_2w_2)\\
&=(q_1,q_2+w_2,3q_1^2,3(q_2+w_2)^2-3w_2^2),
\end{aligned}
\end{equation}
where $q=(q_1,q_2)\in\R^2$ and $w_2\in\R$ are arbitrary, and similarly for the case $w_2=0$.
Therefore, again up to reordering of coordinates,
\begin{equation}\nonumber
\Sigma=(\gamma\times\gamma^-)\cup(\gamma^-\times\gamma)\subset \R^4,
\end{equation}
where
\begin{equation}\nonumber
\gamma^-:=\{(t,u)\mid u\leq 3t^2\}\subset \R^2.
\end{equation}
In this case, $\Sigma$ really has codimension 1 (with singularities). In $\R^4\setminus \Sigma$ there exist points which are not in the image of the map $\Psi$ and points which have four preimages, namely, those in $\gamma^-\times\gamma^-$.

\subsection{Cubic generating functions and dimension 4}

Before specifying to dimension 4, we use a convenient reformulation of finding preimages of the map $\Psi$ for a cubic generating functions. Recall that the map $\Psi:TL\to\R^{2n}$ is given by $(X,W)\mapsto X+W$ which, in terms of a generating function $F:\R^n\to\R$, is the map 
$$
\R^n\times\R^n\ni(q,w)\mapsto (q+w,\nabla F(q)+\nabla^2F(q)w)\in\R^n\times\R^n,
$$ 
as explained in Section \ref{sec:Lagrangian_case_the_wall}. The task is, given $(Q,W)$, to find $(q,w)$ such that
\begin{equation}\nonumber
(Q,W)=(q+w,\nabla F(q)+\nabla^2F(q)w),
\end{equation}
which immediately simplifies to
\begin{equation}\nonumber
W=\nabla F(q)+\nabla^2F(q)[Q-q],
\end{equation}
where we used that $w=Q-q$.

\begin{lemma}\label{lem:preimage_Psi_for_cubic_generating_function}
Let $F$ be a homogeneous cubic polynomial. Then the equation 
\begin{equation}\nonumber
W=\nabla F(q)+\nabla^2F(q)[Q-q]
\end{equation}
is equivalent to
\begin{equation}\nonumber
\nabla F(q+w)=\nabla F(w)+W
\end{equation}
 and to
\begin{equation}\nonumber
\nabla F(w)=\nabla F(Q)-W,
\end{equation}
where $Q=q+w$.
\end{lemma}

\begin{proof}
We combine Euler's formula
\begin{equation}\nonumber
2\nabla F(q)=\nabla^2F(q)q
\end{equation}
with 
\begin{equation}\nonumber
\nabla^3F(\bar q)[\zeta,\xi]=\nabla^2F(\xi)\zeta=\nabla^2F(\zeta)\xi \qquad\forall \bar q,\xi,\zeta\in\R^2
\end{equation}
to conclude that
\begin{equation}\nonumber
\begin{aligned}
2\nabla F(q+w)&=\nabla^2F(q+w)[q+w]\\
&=\nabla^2F(w)[w]+\nabla^2F(q)[q]+\nabla^2F(w)[q]+\nabla^2F(q)[w]\\
&=\nabla^2F(w)[w]+\nabla^2F(q)[q]+2\nabla^2F(q)[w]\\
&=2\nabla F(w)+2\nabla F(q)+2\nabla^2F(q)[w].
\end{aligned}
\end{equation}
This proves the lemma. \proofend
\end{proof}

We will need another lemma. Let $A_1, A_2$ be two non-degenerate symmetric $2\times 2$ matrices. We assume that $A_1$ and $A_2$ are not proportional. Consider the system of equations
\begin{equation}\label{eqn:two_central_conics}
\left\{\;
\begin{aligned}
A_1 w \cdot w&=r_1,\\
A_2 w \cdot w&=r_2,
\end{aligned}\right.
\end{equation}
where $w\in \R^2$ and $\cdot$ denotes the Euclidean inner product on $\R^2$. We are concerned with the number of solutions as the right-hand side $(r_1,r_2)\in \R^2$ varies.

\begin{lemma} \label{lm:central_conics}
The pencil $t_1 A_1 + t_2 A_2,\ (t_1,t_2)\neq (0,0)$, does not contain degenerate matrices if and only if the number of solutions to the system \eqref{eqn:two_central_conics} is two for all $(r_1,r_2)\neq (0,0)$. Otherwise, for a generic choice of $(r_1,r_2)$, the number of solutions is either zero or four.
\end{lemma}

\begin{proof}
We start with two trivial observations. First, if $w$ is a solution to equation \eqref{eqn:two_central_conics} then so is $-w$, and if $(r_1,r_2)\neq (0,0)$ then $w=0$ is not a solution. That is, the number of solutions we are interested in is always even. Moreover, if $A_i$ is positive resp.~negative definite for one (or both) $i=1,2$ then $A_iw\cdot w$ is positive resp.~negative for all $w\neq0$. In particular, for $r_i$ negative resp.~positive there are then no solutions to equation \eqref{eqn:two_central_conics}.

Geometrically, our system of equations describes the intersection of two non-degenerate central conics. As one varies the right-hand sides, these conics vary in homothetic families. Moreover, the vector $A_1 w$ is normal to the conic $A_1 w \cdot w=r_1$ at the point $w$, and likewise for the second conic. In particular, the two conics are tangent at $w$ if and only if the normals $A_1 w$ and $A_2 w$ are collinear. Note that the number of solutions to equation \eqref{eqn:two_central_conics}, i.e., the number of intersection points, can only change when the two conics become tangent, that is, when $A_1 w$ and $A_2 w$ are collinear. In other words, the number of solutions to equation \eqref{eqn:two_central_conics} changes if and only if, for some $(t_1,t_2)\in\R^2$, the matrix $t_1A_1+t_2A_2$ has non-zero kernel.

Generically, two central conics intersect in 0, 2, or 4 points. If one of the conics is an ellipse (i.e., if $A_1$ or $A_2$ is definite), then the number of transverse intersections is either 0 or 4 but never 2. Moreover, 0 and 4 both occur when varying $(r_1,r_2)$, see Figure \ref{twocon}. 
\begin{figure}[hbtp] 
\centering
\includegraphics[height=1.5in]{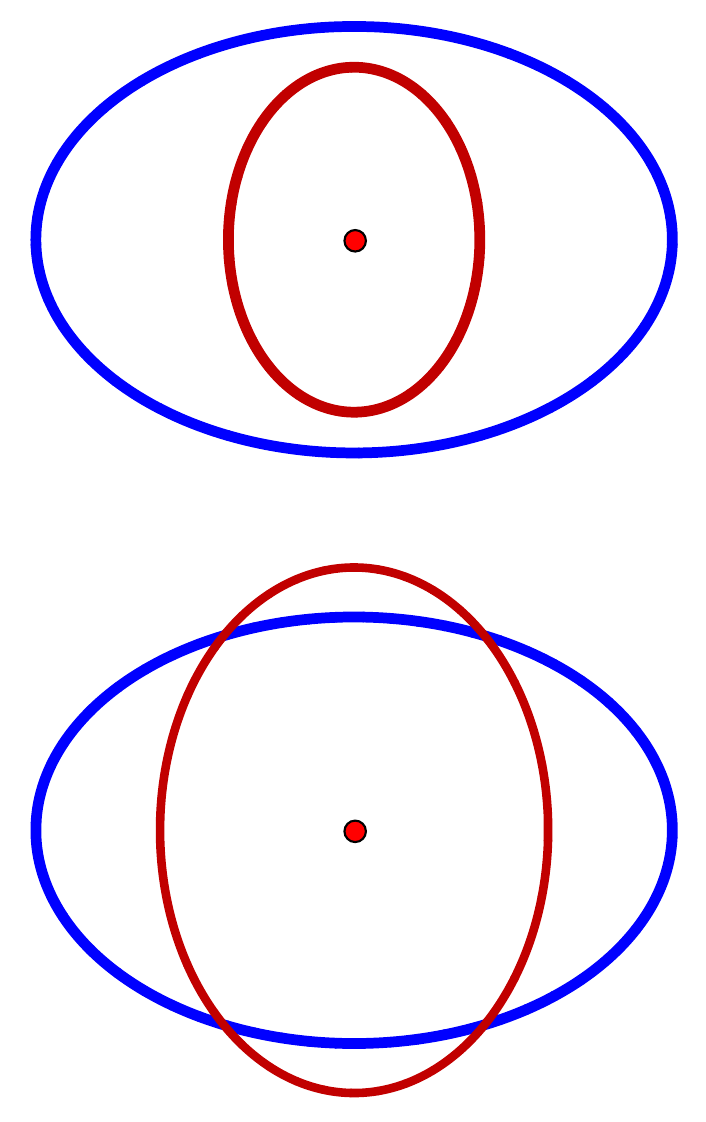}\qquad \qquad
\includegraphics[height=1.5in]{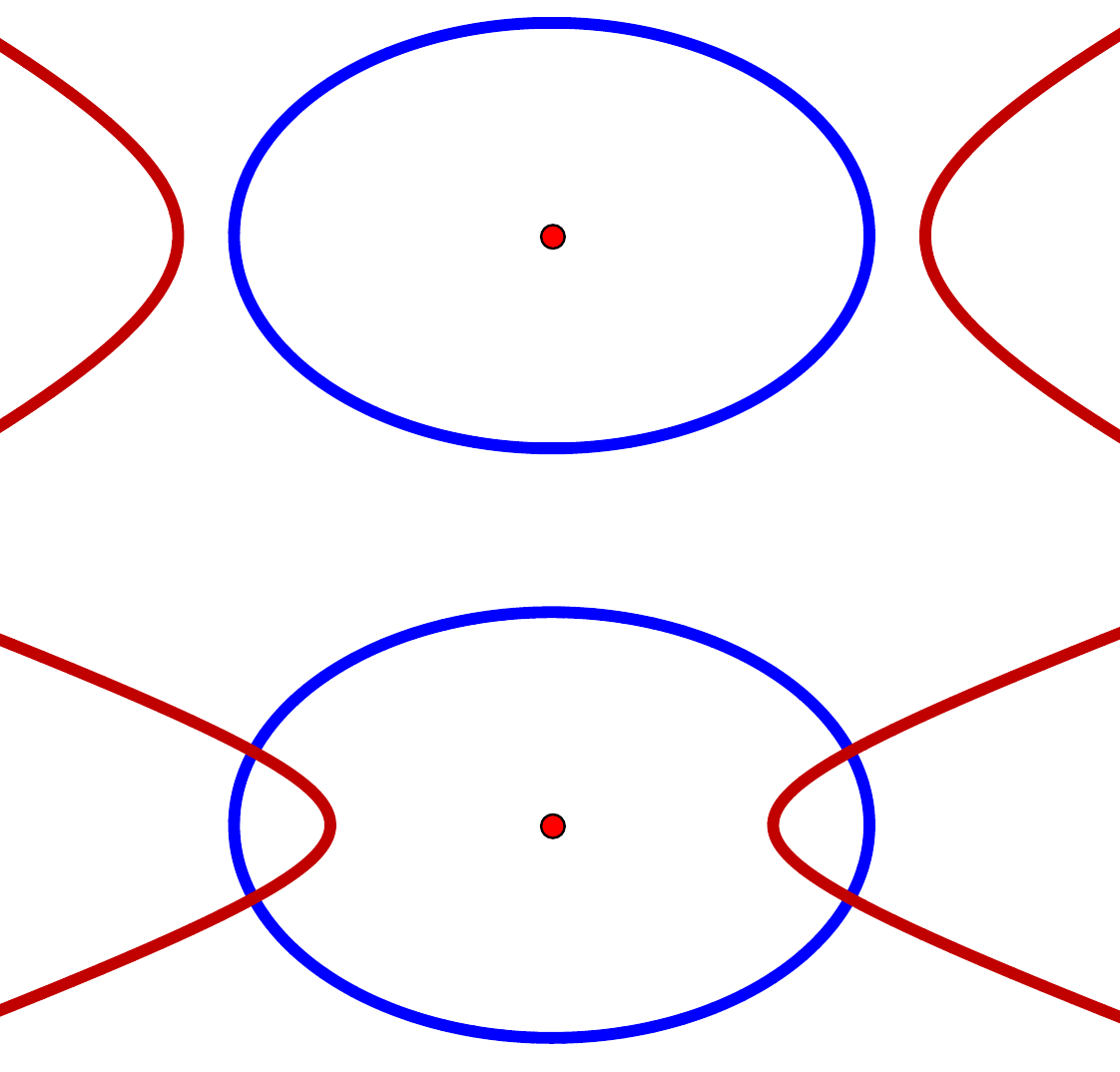}
\caption{Left: two central ellipses; right: a central ellipse and a hyperbola.}
\label{twocon}
\end{figure}

Thus, we are left with the case of two hyperbolas. This has two subcases, depending on the relative locations of the asymptotic lines of each hyperbola to each other. 
\begin{figure}[ht] 
\centering
\includegraphics[height=1.5in]{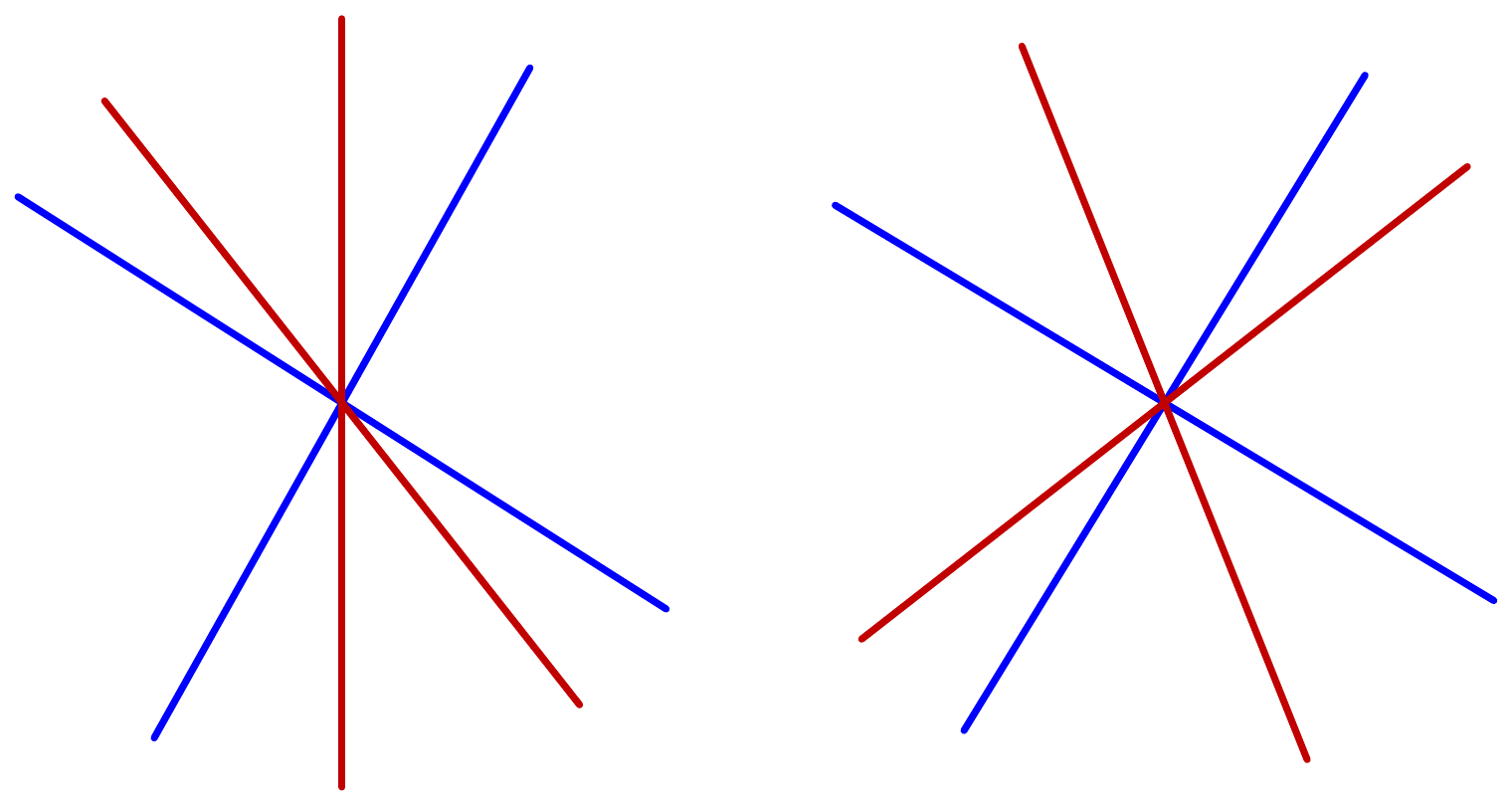}\qquad \qquad
\includegraphics[height=1.5in]{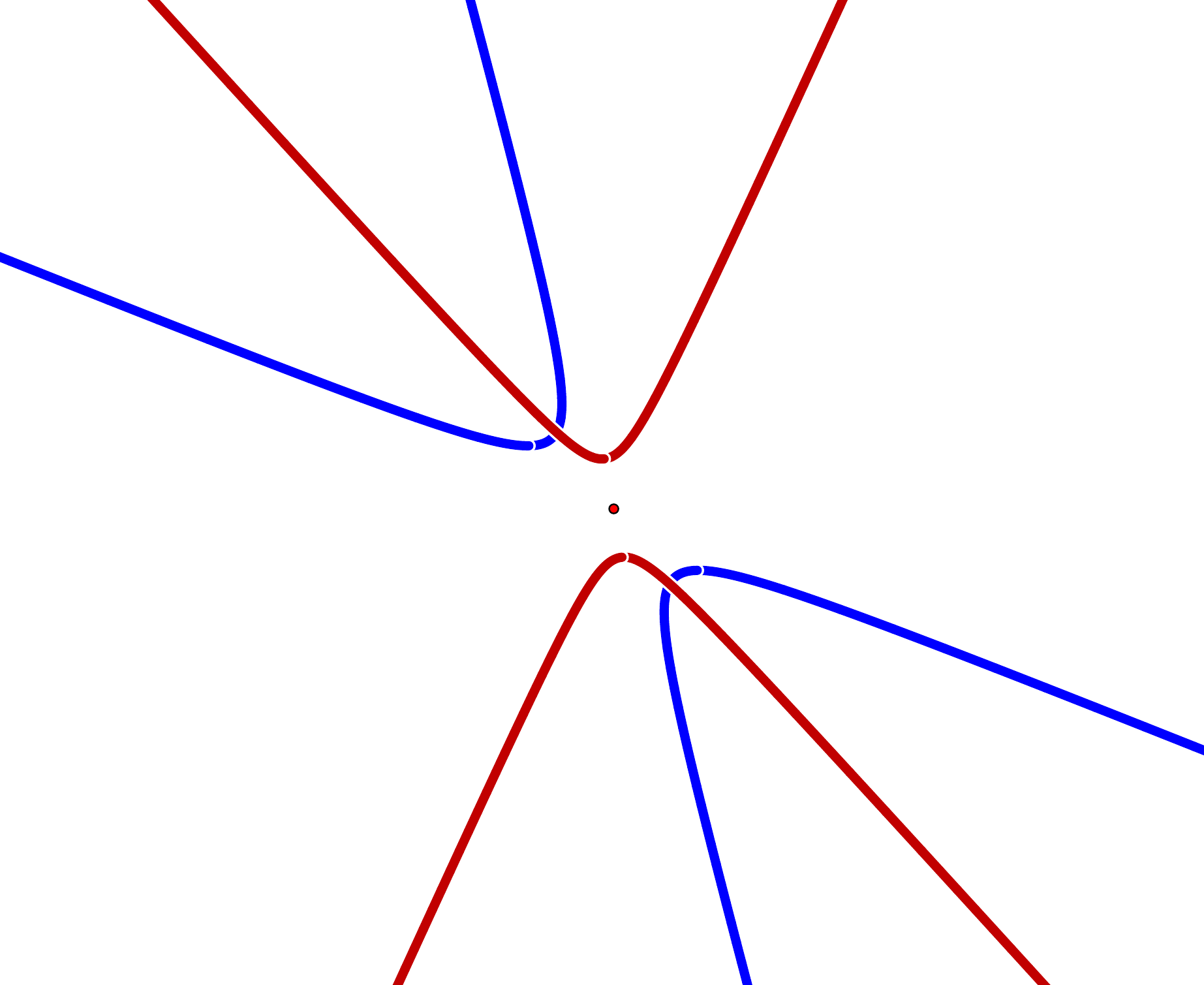}
\caption{Left: nested and interlacing pairs of asymptotes; right: interlacing hyperbolas.}
\label{asymptotes}
\end{figure}
They can either be nested, i.e., the two asymptotic lines of one hyperbola partition $\R^2$ into four sectors, two of which contain the asymptotic lines of the other hyperbola, or, otherwise, the asymptotic lines are interlacing, see Figure \ref{asymptotes}.

If the two hyperbolas are nested, then depending on the choices of $(r_1,r_2)$, there are 0 or 4 intersection points and both can be realized, see Figure \ref{cases}.
\begin{figure}[hbtp] 
\centering
\includegraphics[height=1.5in]{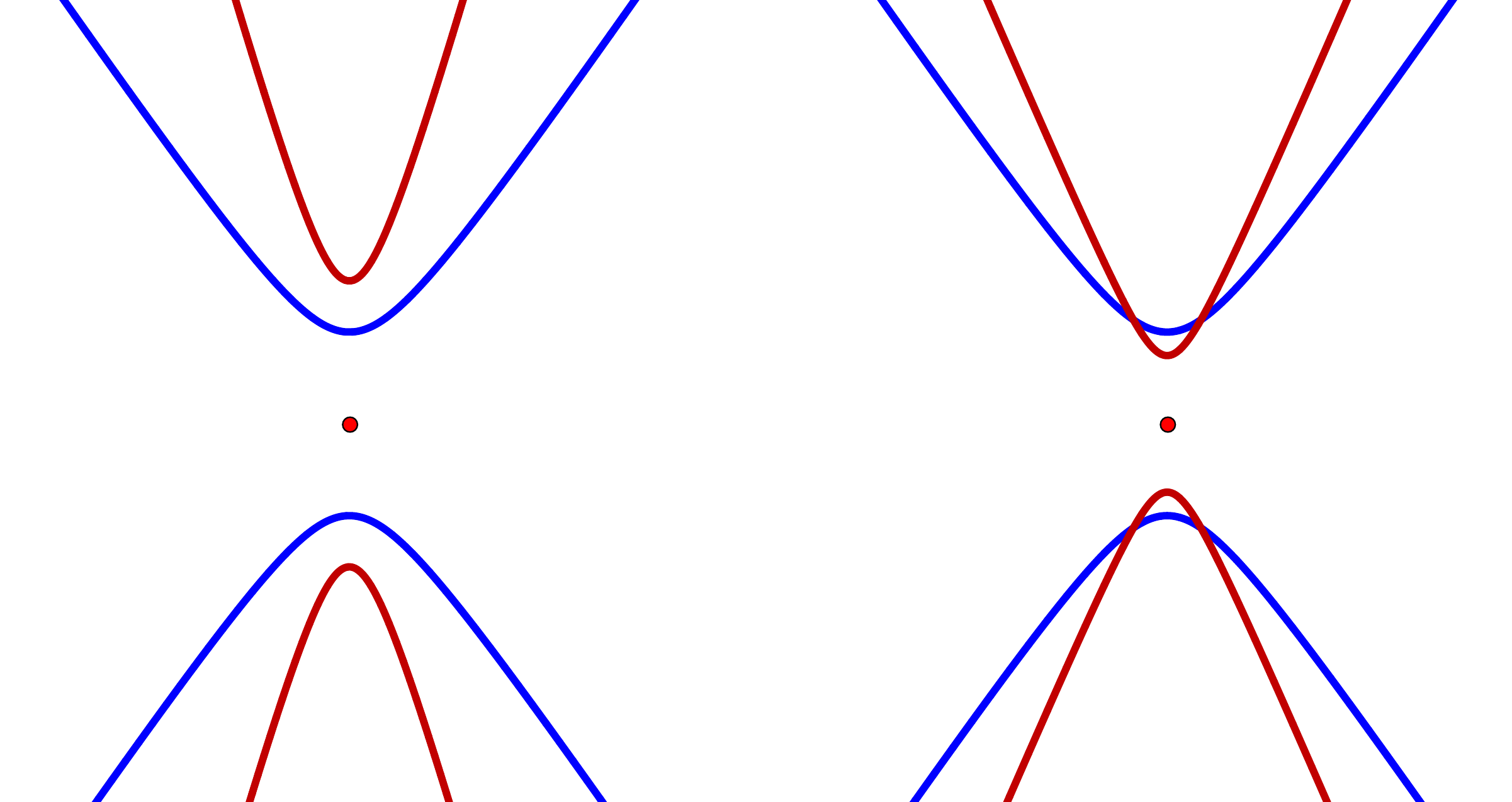}
\caption{Nested hyperbolas: four intersections on the right and none on the left.}
\label{cases}
\end{figure}

In the interlacing case, the number of solutions to equation \eqref{eqn:two_central_conics} is always two (unless $(r_1,r_2)=(0,0)$). Indeed, it does not change as the right-hand side varies since the two hyperbolas always remain transverse to each other, see Figure \ref{asymptotes} again. Accordingly, the normal vectors $A_1 w$ and $A_2 w$ are not collinear for all nonzero $w$. In other words, the matrix $t_1 A_1 + t_2 A_2$ has trivial kernel for all $(t_1,t_2)\neq (0,0)$. This finished the proof. \proofend
\end{proof}

Motivated by the discussion from the previous section we consider now cubic generating functions in two variables. That is, we consider the generating function
\begin{equation}\label{eq:homogeneous_cubic_generating_function_dim_4}
F(q_1,q_2)=aq_1^3+bq_1^2q_2+cq_1q_2^2+dq_2^3
\end{equation}
for some $a,b,c,d\in\R$ and the corresponding Lagrangian submanifold
\begin{equation}\nonumber
L=\big\{(q_1,q_2,3aq_1^2+2bq_1q_2+cq_2^2,bq_1^2+2cq_1q_2+3dq_2^2)\mid(q_1,q_2)\in\R^2\big\}
\end{equation}
in $\R^4$. Associated are the set $\Delta$ of critical points and the set $\Sigma$ of critical values of the map $\Psi:TL\to\R^4$ resp.~$\Psi:\R^2\times\R^2\to\R^2\times\R^2$.

\begin{proposition}\label{prop:abcd_example_and_wall} 
The following statements are equivalent:
\begin{enumerate}   \renewcommand{\theenumi}{\roman{enumi}}\renewcommand{\labelenumi}{\rm (\theenumi)} 
\item The coefficients $a,b,c,d$ of $F$ satisfy the inequality 
\begin{equation}\label{eq:no_wall_condition_cubic_examples}
18abcd+b^2c^2>4ac^3+4b^3d+27a^2d^2.
\end{equation}
\item $\Delta=L$, where $L\subset TL$ denotes the zero-section, and then, in particular, $\Sigma=L\subset\R^4$.
\item Every regular value of the map $\Psi:TL\to\R^4$ has multiplicity 2.
\end{enumerate}
The opposite inequality 
\begin{equation}\label{eq:no_wall_condition_cubic_examples_other_way_round}
18abcd+b^2c^2<4ac^3+4b^3d+27a^2d^2,
\end{equation}
is equivalent to the multiplicity of $\Psi$ being $0$ and $4$. In this case, $\Delta\neq L$.
\end{proposition}

\begin{remark}$ $
{\rm
\begin{itemize}  
\item Two of the examples in the previous section exemplify the two cases in the proposition. Indeed, $F=q_1^2q_2+q_1q_2^2$ corresponds to $b=c=1$ and $a=d=0$ and inequality \eqref{eq:no_wall_condition_cubic_examples} becomes $1>0$. Moreover, $F=q_1^3+q_2^3$ corresponds to $a=d=1$ and $b=c=0$ and inequality \eqref{eq:no_wall_condition_cubic_examples_other_way_round} becomes $0<27$. 
\item The second example, the affine Lagrangian subspace, corresponds basically to the case $F=0$, i.e.,~neither inequality holds. In fact, the linear and quadratic terms in any polynomial generating function can be removed by applying an affine symplectic transformation to the Lagrangian submanifold. Therefore, the cubic case is the first interesting case and we may assume without loss of generality that the generating function is homogeneous as in \eqref{eq:homogeneous_cubic_generating_function_dim_4}. 

\item If inequality \eqref{eq:no_wall_condition_cubic_examples} holds then, after one initial choice, the outer symplectic billiard correspondence defines a map as follows. Every point $Z\in\R^{2n}\setminus L$ has precisely two preimages under $\Psi$. If we  write $Z=X+W$ with $X\in L$ and $W\in T_{X}L$, then the outer symplectic billiard map maps $Z=X+W$ to $Z'=X-W$. Now, $Z'=X-W=X'+W'$ for a unique choice of $X'\in L$ and $W'\in T_{X'}L$. Thus, $Z'$ is mapped to $X'-W'$, etc.

\end{itemize}
} 
\end{remark}

Let us prove Proposition \ref{prop:abcd_example_and_wall}.

\begin{proof}
As explained above, we always have an inclusion of the zero-section $L\subset\Delta\subset TL$ which, in terms of the generating function $F$, becomes $L=\{w=(w_1,w_2)=0\}\subset\Delta\subset\R^2\times\R^2$. Let us list again the derivatives of $F$:
\begin{equation}\nonumber
\begin{aligned}
\nabla F(q)&=(3aq_1^2+2bq_1q_2+cq_2^2,bq_1^2+2cq_1q_2+3dq_2^2),\\[2ex]
\nabla^2F(q)w
&=
\begin{pmatrix}
6aq_1+2bq_2 & 2bq_1+2cq_2\\
2bq_1+2cq_2 & 2cq_1+6dq_2
\end{pmatrix}
\begin{pmatrix}
w_1\\
w_2
\end{pmatrix},\\[2ex]
\nabla^3F(q)[w,\zeta]&=
\begin{pmatrix}
6aw_1+2bw_2 & 2bw_1+2cw_2\\
2bw_1+2cw_2 & 2cw_1+6dw_2
\end{pmatrix}
\begin{pmatrix}
\zeta_1\\
\zeta_2
\end{pmatrix}
=\nabla^2F(w)\zeta.
\end{aligned}
\end{equation}
We recall from Lemma \ref{lem:description_of_wall_Lagrangian_case} that $\Delta$ is the set of points $(q,w)\in\R^2\times\R^2$ such that the linear map $\R^2\ni\zeta\mapsto \nabla^3F(q)[\zeta,w]\in\R^2$ is singular. Its determinant, dropping a factor of 4, is
\begin{equation}\label{eq:abcd_quadratic_polynomial_w1_w2}
(3ac-b^2)w_1^2+(9ad-bc)w_1w_2+(3bd-c^2)w_2^2.
\end{equation}
Of course, $w_1=w_2=0$ is a solution, confirming again $\{w=0\}\subset\Delta$. If $w_2\neq0$, we may reduce \eqref{eq:abcd_quadratic_polynomial_w1_w2} to the quadratic equation for $z=\frac{w_1}{w_2}$ given by
\begin{equation*} 
\begin{aligned}
(3ac-b^2)z^2+(9ad-bc)z+(3bd-c^2)=0.
\end{aligned}
\end{equation*}
This quadratic equation has no real solution if and only if
\begin{equation*}\label{eq:abcd_quadratic_polynomial_inequality}
(9ad-bc)^2<4(3ac-b^2)(3bd-c^2).
\end{equation*}
Expanding and reordering the terms gives inequality \eqref{eq:no_wall_condition_cubic_examples}. The case $w_1\neq0$ leads to the same inequality. Thus, we proved that inequality \eqref{eq:no_wall_condition_cubic_examples} is equivalent to $\Delta=L$, that is, (i) $\Longleftrightarrow$ (ii). 

\medskip

To show the equivalence of (i), resp.~(ii), to statement (iii), we recall from Lemma \ref{lem:preimage_Psi_for_cubic_generating_function} that finding a preimage of $(Q,W)\in\R^2\times\R^2$ under $\Psi$ is equivalent to finding $(q,w)\in\R^2\times\R^2$ satisfying
\begin{equation}\nonumber
\nabla F(w)=\nabla F(Q)-W
\end{equation}
with $q=Q-w$. For $F=aq_1^3+bq_1^2q_2+cq_1q_2^2+dq_2^3$ this leads to the two quadratic equations
\begin{equation}\label{eqn:intersection_central_conics}
\left\{\;
\begin{aligned}
3aw_1^2+2bw_1w_2+cw_2^2&=r_1,\\
bw_1^2+2cw_1w_2+3dw_2^2&=r_2,
\end{aligned}
\right.
\end{equation}
where $(r_1,r_2)=\nabla F(Q)-W\in\R^2$. In particular, we may apply Lemma \ref{lm:central_conics}.

Note that, by varying $(Q,W)\in\R^2\times\R^2$ arbitrarily, the value of $(r_1,r_2)=\nabla F(Q)-W$ ranges over all of $\R^2$. 

We claim that $(Q,W)$ is a regular value of $\Psi$ if and only if the two conics in \eqref{eqn:intersection_central_conics} intersect transversely. Indeed, the conics are tangent when their normals are proportional, that is, when the matrix 
\begin{equation} \label{eq:samemat}
\begin{pmatrix}
3aw_1+bw_2 & bw_1+cw_2\\
bw_1+cw_2 & cw_1+3dw_2
\end{pmatrix}
\end{equation}
is degenerate. As we saw above, this condition describes the critical points of the map $\Psi$. 

According to Lemma \ref{lm:central_conics}, the system (\ref{eqn:intersection_central_conics}) has a constant number of solutions (which is equal to two) if and only if matrix (\ref{eq:samemat}) is nondegenerate for all $(w_1,w_2) \neq (0,0)$, and as we know, this is equivalent to (ii) and to (i). 

The case of the opposite inequality (\ref{eq:no_wall_condition_cubic_examples_other_way_round}) corresponds to the three other cases of Lemma \ref{lm:central_conics}, where the number of solution for a generic right-hand side is either zero or four.
\proofend
\end{proof}

\subsection{General Lagrangian submanifolds}

We begin by considering a cubic generating function $F:\R^n\to\R$ in arbitrary dimensions and the associated Lagrangian submanifold $L\subset\R^{2n}$.

\begin{proposition}\label{prop:ruled_means_nabla_F_w_w_0}$ $
\begin{enumerate}  
\item The following three assertions are equivalent. 
\begin{enumerate}  
\item $L$ is the union of parallel lines. In particular, $L$ is ruled.
\item $L$ contains one line, i.e., in terms of $\Psi$
	\begin{equation}\nonumber
	\begin{aligned}
	\exists q\in\R^n\;\exists w\in\R^n\setminus\{0\}\;\forall t\in\R\;:\; \Psi(q,tw)\subset L.
	\end{aligned}
	\end{equation}  
\item There exists $w\in\R^n\setminus\{0\}$ with $\nabla^2F(w)w=0$.
\end{enumerate}
\item Assume that $n=2$. Then there exists a $w\in\R^2\setminus\{0\}$ with $\nabla^2F(w)w=0$ if and only if
\begin{equation}\nonumber
18abcd+b^2c^2=4ac^3+4b^3d+27a^2d^2
\end{equation}
and
$$
b^2>3ac,\ c^2>3bd,
$$
where $F=aq_1^3+bq_1^2q_2+cq_1q_2^2+dq_2^3$ is as in the previous section.
\end{enumerate}
\end{proposition}

\begin{remark}
{\rm
Certainly, the statement that $L$ contains a line is independent of a generating function. We can rephrase this as the existence of a vector $W\in\R^{2n}$ such that $X+\R W\subset L$, and of course, then $W\in T_XL$.

If any of the equivalent assertions in 1.~is true, the equality $\Delta=TL$ follows from combining Lemma \ref{lem:description_of_wall_Lagrangian_case} with $\nabla^3F(q)[w,\zeta]=\nabla^2F(w)\zeta$. In particular, then $\Sigma=\Psi(\Delta)=\Psi(TL)$, and the outer symplectic billiard correspondence is not well-defined anywhere.
} 
\end{remark}

\begin{proof}
It is enough to prove the equivalence between 1.(b) and 1.(c) since the latter condition is actually independent of the point $q$. That is, then 1.(b) holds for all $q$ with the same vector $w$ which is 1.(a).

Let us assume that $L$ contains a line, i.e., we can find $q,w\in\R^n$ with $w\neq0$ such that
\begin{equation}\nonumber
\Psi(q,tw)=\big(q+tw,\nabla F(q)+t\nabla^2F(q)w\big)\in L\quad\forall t\in\R.
\end{equation}
In particular, we find $Q(t)\in\R^n$ such that
\begin{equation}\nonumber
\big(q+tw,\nabla F(q)+t\nabla^2F(q)w\big)=\big(Q(t),\nabla F(Q(t))\big).
\end{equation}
We immediately obtain $Q(t)=q+tw$. The Taylor expansion 
\begin{equation}\nonumber
\nabla F(q+tw)=\nabla F(q) +t \nabla^2F(q)w+\tfrac12t^2\nabla^3F(q)[w,w]
\end{equation}
is exact since $F$ is cubic, and it follows that
\begin{equation}\nonumber
\nabla^3F(q)[w,w]=0.
\end{equation}
Using again the equality $\nabla^3F(q)[w,w]=\nabla^2F(w)w$, we see that if $L$ contains a line then $\nabla^2F(w)w=0$ for some $w\neq0$. The above computation also shows that $\nabla^2F(w)w=0$ implies that $\Psi(q,tw)\in L$ for all $t\in\R$. This proves point 1.~in the proposition.

To show 2.~we recall from the proof of Proposition \ref{prop:abcd_example_and_wall} that
\begin{equation}\nonumber
\nabla^2F(w)w=\begin{pmatrix}
6aw_1+2bw_2 & 2bw_1+2cw_2\\
2bw_1+2cw_2 & 2cw_1+6dw_2
\end{pmatrix}
\begin{pmatrix}
w_1\\
w_2
\end{pmatrix}
=0
\end{equation}
are the two real equations
\begin{equation}\nonumber
\left\{\;
\begin{aligned}
3aw_1^2+2bw_1w_2+cw_2^2&=0\\
bw_1^2+2cw_1w_2+3dw_2^2&=0.
\end{aligned}
\right.
\end{equation}
Again $w_1=w_2=0$ is a solution and if $w_2\neq0$ this systems reduces to the quadratic equations in $z=\frac{w_1}{w_2}$ given by
\begin{equation}\nonumber
\left\{\;
\begin{aligned}
3az^2+2bz+c&=0\\
bz^2+2cz+3d&=0.
\end{aligned}
\right.
\end{equation}
A classical fact is that two complex polynomials in one variable have a common root if and only if their resultant is zero. The resultant of the two equations from above is by definition
\begin{equation}\nonumber
\det
\begin{pmatrix}
3a & 0 & b & 0 \\
2b & 3a & 2c & b \\
c & 2b & 3d & 2c \\
0 & c  & 0 & 3d 
\end{pmatrix}=3\big(18abcd+b^2c^2-(4ac^3+4b^3d+27a^2d^2)\big).
\end{equation}
Since we are working over reals, for the quadratic equations to have solutions we also need their discriminants to be positive:
$$
b^2 - 3ac > 0,\ c^2 - 3bd > 0. 
$$ 
We conclude that $\nabla^2F(w)w=0$ for some $w\neq0$ is equivalent to the claimed equation $18abcd+b^2c^2=4ac^3+4b^3d+27a^2d^2$ and the inequalities $b^2>3ac, c^2>3bd$.
\proofend 
\end{proof}


\begin{remark} \label{Bezout}
{\rm
The expression
$$
D=18abcd+b^2c^2-4b^3d-4ac^3-27a^2d^2
$$
has appeared several times by now. Let us comment on this fact. In the following discussion we work over the complex numbers.

The expression  $D$ is the discriminant of the cubic polynomial $az^3+bz^2+cz+d$, hence $D=0$ if and only if this polynomial has a multiple root. Accordingly, the cubic form $F(x,y)=ax^3+bx^2y+cxy^2+dy^3$ is a product of three linear forms, and $D=0$ if and only if $F$ can be written as 
$F=f^2 g$, where $f$ and $g$ are linear forms.

In the proof of Proposition \ref{prop:ruled_means_nabla_F_w_w_0}, we dealt with non-zero solutions of the equation $\nabla^2F (w) w =0$, where $w=(x,y)$. Recall that $F$ is a cubic form, and its partial derivatives are quadratic forms. By Euler's formula, 
$$
\left\{\;
\begin{aligned}
x F_{xx}(x,y)+yF_{xy}(x,y)&=2F_x(x,y)\\
xF_{xy}(x,y)+yF_{yy}(x,y)&=2F_y(x,y).
\end{aligned}
\right.
$$
Hence, 
the system of equations $\nabla^2F(x,y)(x,y)=0$ is equivalent to $F_x(x,y)=F_y(x,y)=0$.

Now, if $F=f^2g$, then $F_x(x,y)=F_y(x,y)=0$ for all $(x,y)$ such that $f(x,y)=0$, i.e., on the line $\{f=0\}$. 
On the other hand, if $F=fgh$ with non-proportional linear factors, then the vector equation $\nabla F(x,y)=0$  
has only the trivial solution $(x,y)=(0,0)$.  Indeed, in the latter case, a linear change of variables makes it possible to assume that
$$
f(x,y)=x,\ g(x,y)=y,\ h(x,y)=x+y.
$$
Then the system becomes $2xy+y^2=0=x^2+2xy$. Each equation describes a union of two lines through the origin, and all four lines are different. Hence, the only solution is the origin. 

In the proof of Proposition \ref{prop:abcd_example_and_wall}, we wanted to know when the quadratic equation $\det \nabla^2F=F_{xx}F_{yy}-F_{xy}^2=0$ had no real solutions, other than $(0,0)$. For that, consider the discriminant of the quadratic equation we obtain from $F_{xx}F_{yy}-F_{xy}^2=0$ by considering the variable $y/x$. There are no real solutions, other than $(0,0)$, if and only if this discriminant is negative. In the proof of Proposition \ref{prop:abcd_example_and_wall} we show that this discriminant is, up to a multiplicative factor, equal to $D$.

Let us add that if $F=f^2g$, then 
$$
F_{xx}F_{yy}-F_{xy}^2 = -4 f^2 (f_xg_y-f_yg_x)^2.
$$
That is, in this case, $\det \nabla^2F$, considered as a quadratic polynomial in $y/x$, indeed has a double root given by the equation $f(x,y)=0$. 



}
\end{remark}

The rest of the section deals with general Lagrangian submanifolds in $\R^{2n}$. We stress that, in the following, we do \textbf{not} assume that the Lagrangian submanifold is given by a cubic generating function. 

\begin{theorem}
Let $L\subset\R^{2n}$ be a Lagrangian submanifold and consider the set $\Delta\subset TL$ of critical points of the map $\Psi:TL\to\R^{2n}$, $\Psi(X,W)=X+W$. We assume that $\Delta$ intersects some tangent fiber $T_XL$ only in the zero section, i.e., there exists $X\in L$ with
\begin{equation}\nonumber
\Delta\cap T_XL=\{0\}.
\end{equation}
Then $\dim L =1,2$.
\end{theorem}

\begin{remark}
{\rm
For $\dim L=2$, Example 3.2 / 3.5 from \cite{FT2}, i.e., the first example in Section \ref{subsect:3examples}, is a Lagrangian submanifold having the property $\Delta\cap TL=$ zero section, i.e., \textbf{all} tangent fibers intersect $\Delta$ only in zero. Any embedded closed curve with strictly positive curvature is an example with $\dim L=1$ and the same property.
}
\end{remark}

\begin{proof}
The assumption of the theorem is clearly local and we therefore assume without loss of generality that $L$ is globally given by a generating function $F:\R^n\to\R$.

In terms of the generating function the assumption becomes that there exists $q_0\in\R^n$ such that
\begin{equation}\nonumber
\Delta\cap T_{(q_0,\nabla F(q_0))}L=\{0\}.
\end{equation}
We recall from Lemma \ref{lem:description_of_wall_Lagrangian_case} that $(q,w)\in\Delta$ if and only if $\R^n\ni\zeta\mapsto \nabla^3F(q)[\zeta,w]\in\R^n$ does not have full rank. Therefore, by assumption we know that
\begin{equation}\nonumber
\R^n\ni\zeta\mapsto \nabla^3F(q_0)[\zeta,w]\in\R^n
\end{equation}
has full rank for all $w\in\R^n\setminus\{0\}$. This means that the symmetric bilinear map
\begin{equation}\nonumber
\begin{aligned}
\R^n\times\R^n&\to\R^n\\
(\zeta,w)&\mapsto\nabla^3F(q_0)[\zeta,w]
\end{aligned}
\end{equation}
has no zero-divisors. Indeed, assume that there is $(\zeta,w)$ with $\nabla^3F(q_0)[\zeta,w]=0$. If $w\neq0$ then, by our assumption, $\zeta=0$.

In 1940 Heinz Hopf gave in \cite{Hopf} a very elegant argument that the existence of a symmetric (or commutative), bilinear map $\R^n\times\R^n\to\R^n$ without zero divisors implies $n=1,2$, see \cite{Ja} or \cite{Ha} for a modern treatment. For convenience, we present Hopf's argument below. \proofend 
\end{proof}

\begin{theorem}[H. Hopf]
Let $f:\R^n\times\R^n\to\R^{n+k}$ be bilinear, symmetric, and free of zero-divisors then there exists an embedding $\bar h:\RP^{n-1}\to S^{n+k-1}$.
\end{theorem}

\begin{proof}
The following proof is essentially contained in \cite{Hopf}, see also \cite{Ja}. Consider the map
\begin{equation}\nonumber
\begin{aligned}
h:S^{n-1}\times S^{n-1}&\to S^{n+k-1}\\
(x,y)&\mapsto \frac{f(x,y)}{||f(x,y)||}.
\end{aligned}
\end{equation}
This map is well-defined since $f$ is free of zero-divisors. Since $f$ is bilinear, it induces a map
\begin{equation}\nonumber
\begin{aligned}
\bar h:\RP^{n-1}&\to S^{n+k-1}\\
[x]&\mapsto \bar h([x]):=h(x,x).
\end{aligned}
\end{equation}
Assume that $\bar h([x])=\bar h([y])$, i.e., $f(x,x)=t^2f(y,y)$ for some $t\in\R\setminus\{0\}$. Then, using bilinearity and symmetry, we see
\begin{equation}\nonumber
\begin{aligned}
f(x+ty,x-ty)=f(x,x)-t^2f(y,y)=0.
\end{aligned}
\end{equation}
Since $f$ does not have zero divisors we see that $x=\pm ty$, $t\in\R\setminus\{0\}$, and conclude $[x]=[y]$. In other words, $\bar h$ is injective and thus defines the desired embedding $h:\RP^{n-1}\to S^{n+k-1}$. \proofend
\end{proof}

\begin{corollary}
Let $f:\R^n\times\R^n\to\R^{n}$ be bilinear, symmetric, and free of zero-divisors then $n=1,2$.
\end{corollary}

\begin{proof}
There exists an embedding $h:\RP^{n-1}\to S^{n-1}$ if and only if the dimension satisfies $n=1,2$. \proofend
\end{proof}

\subsection{Integrability of quadratic Lagrangian submanifolds}

We directly start with the main theorem of the section title.

\begin{theorem} \label{thm:int}
Let $L\subset\R^{2n}$ be a Lagrangian submanifold admitting a cubic generating function $F:\R^n\to\R$. Then the outer symplectic billiard relation associated with $L$ is completely integrable. More precisely, the $n$ components of the map $\R^{2n}\ni(Q,P)\mapsto P-\nabla F(Q)\in\R^n$ are Poisson commuting and invariant under the outer symplectic billiard relation.
\end{theorem}

\begin{proof}
The functions
\begin{equation}\nonumber
G_1(Q, P):=P_1-\frac{\partial F}{\partial Q_1}(Q),\;\cdots\;,G_n(Q,P):=P_n-\frac{\partial F}{\partial Q_n}(Q)
\end{equation}  
Poisson commute since the second partial derivatives of $F$ commute, indeed
\begin{equation}\nonumber
\begin{aligned}
dG_j(X_{G_k})&=\left[dP_j-\sum_{r=1}^n \frac{\partial^2F}{\partial Q_r\partial Q_j}dQ_r\right](X_{G_k})\\
&=\left[dP_j-\sum_{r=1}^n \frac{\partial^2F}{\partial Q_r\partial Q_j}dQ_r\right]\left(\partial_{Q_k}+\sum_{s=1}^n \frac{\partial^2F}{\partial Q_s\partial Q_k}\partial_{P_s}\right)\\
&=-\frac{\partial^2F}{\partial Q_k\partial Q_j}+\frac{\partial^2F}{\partial Q_j\partial Q_k}=0.
\end{aligned}
\end{equation}
It remains to show that these functions are invariant under the outer symplectic billiard relation. For that we consider two points $A,B\in\R^{2n}$ in outer symplectic billiard relation, that is, there is $X\in L$ and $W\in T_XL$ such that
\begin{equation*}
A=X+W=\Psi(X,W)\quad\text{and}\quad B=X-W=\Psi(X,-W),
\end{equation*}
which, in terms of the generating function, becomes
\begin{equation*}
\begin{aligned}
A=(q+w,\nabla F(q)+\nabla^2F(q)w)\quad\text{and}\quad B=(q-w,\nabla F(q)-\nabla^2F(q)w).
\end{aligned}
\end{equation*}
Using the (exact) Taylor expansion 
$$\nabla F(q\pm w)=\nabla F(q)\pm\nabla^2F(q)w+\frac12\nabla^3F(q)[w,w]$$
of the cubic function $F$, we see that the values of the map $P-\nabla F(Q)$ at the points $A$ and $B$ agree and are equal to
\begin{equation}\nonumber
\big(\nabla F(q)\pm\nabla^2F(q)w-\nabla F(q\pm w)\big)= -\tfrac12\nabla^3F(q)[w,w].
\end{equation}
This completes the proof. \proofend  
\end{proof}

\hrulefill

\bigskip
\noindent Peter Albers: 
Institute for Mathematics,
Heidelberg University,
69120 Heidelberg,
Germany;
peter.albers@uni-heidelberg.de\\

\noindent Ana Chavez Caliz:
Institute for Mathematics,
Heidelberg University,
69120 Heidelberg,
Germany;
anachavezcaliz@gmail.com\\

\noindent Serge Tabachnikov:
Department of Mathematics,
Pennsylvania State University,
University Park, PA 16802,
USA;
tabachni@math.psu.edu


\begin{thebibliography}{99}

\bibitem{AT} P. Albers, S. Tabachnikov. {\it Symplectically convex and symplectically star-shaped curves: a variational problem}.  Symplectic geometry -- a Festschrift in honor of Claude Viterbo's 60th birthday, 5--28, Birkh\"auser/Springer, Cham, 2022.

\bibitem{Ch} M. Chaperon.  {\it Une id\'ee du type ``g\'eod\'esiques bris\'ees" pour les syst\'emes hamiltoniens.}  C. R. Acad. Sci. Paris S\'er. I Math. 298 (1984), no. 13, 293--296. 

\bibitem{DT} F. Dogru, S. Tabachnikov. {\it Dual billiards}. Math. Intelligencer {\bf 27}, No 4 (2005), 18--25.

\bibitem{FT1} M. Farber, S. Tabachnikov. {\it Topology of cyclic configuration spaces and periodic trajectories of multi-dimensional billiards.} Topology {\bf 41} (2002),  553--589. 

\bibitem{FT2} D. Fuchs, S. Tabachnikov. {\it On  Lagrangian tangent sweeps and Lagrangian outer billiards}. Geom. Dedicata {\bf 182} (2016), 203--213.

\bibitem{Gi} A. Givental. {\it Periodic mappings in symplectic topology.}  Funct. Anal. Appl. {\bf 23} (1989),  287--300. 

\bibitem{Ha} A. Hatcher. {\it  Algebraic topology.} Cambridge Univ. Press, Cambridge, 2002.

\bibitem{Hopf} H. Hopf. {\it Systeme symmetrischer {B}ilinearformen und euklidische
              {M}odelle der projektiven {R}\"{a}ume.} Vierteljschr. Naturforsch. Ges. Zürich {\bf 85} (1940), 165--177.

\bibitem{Ja} I. M. James. {\it Euclidean models of projective spaces.} Bull. London Math. Soc {\bf 3} (1971), 257--276.

\bibitem{Mo1} J. Moser. {\it Stable and random motions in dynamical systems.} Ann.
of Math. Studies 77, Princeton Univ. Press, Princeton, NJ, 1973.

\bibitem{Mo2} J. Moser. {\it Is the solar system stable?} Math. Intelligencer {\bf 1} (1978), 65--71.

\bibitem{Mo3} J. Moser. {\it Various Aspects of Integrable Hamiltonian Systems.} In: Marchioro, C. (eds) Dynamical Systems. C.I.M.E. Summer Schools, vol 78. Springer.


\bibitem{Neu} B. Neumann. {\it Sharing ham and eggs.} Iota, Manchester University,
1959.

\bibitem{Ta1} S. Tabachnikov. {\it On the dual billiard problem.} Adv. Math. {\bf 115} (1995),  221--249. 

\bibitem{Ta2} S. Tabachnikov. {\it Asymptotic dynamics of the dual billiard transformation.} J. Statist. Phys. {\bf 83} (1996),  27--37. 

\bibitem{Ta3} S. Tabachnikov. {\it  On three-periodic trajectories of multi-dimensional dual billiards.} Algebr. Geom. Topol. {\bf 3} (2003), 993--1004.

\bibitem{Ta} S. Tabachnikov. {\it Geometry and billiards.} Amer. Math. Soc., Providence, RI, 2005.


\end{thebibliography}
\end{document}